\documentclass[12pt]{article}
\usepackage{amsmath,latexsym,amsxtra,enumerate,
color, cancel}
\usepackage{bbm}
\usepackage{amsthm}
\usepackage[english]{babel}
\usepackage[usenames,dvipsnames,svgnames,table]{xcolor}
\usepackage{times, amsfonts, amssymb, amsthm, amscd, graphicx,stmaryrd}  
\RequirePackage[colorlinks,citecolor=blue,urlcolor=blue]{hyperref}

\newcommand{\D}{\mathcal{D}}
\newcommand{\B}{\mathcal{B}}
\newcommand{\F}{\mathcal{F}}
\newcommand{\R}{\mathbb{R}}

\newcommand{\N}{\mathbb{N}}
\newcommand{\h}{\mathcal{H}}

\newcommand{\E}{\mathbb{E}}
\newcommand{\PP}{\mathbb{P}}
\newcommand{\one}{{\rm 1\hspace{-0.11cm}I}}
\newcommand{\half}{\frac{1}{2}}
\newcommand{\tx}{{\tt x}}

\newcommand{\Ltwo}{\mathrm{L}_2}
\newcommand{\A}{\bf A}
\newcommand{\kla}{\left ( }
\newcommand{\mer}{\right ) }
\newcommand{\non}{\nonumber }

\theoremstyle{plain}
\newtheorem{theorem}{Theorem}[section]
\newtheorem{lemma}[theorem]{Lemma}

\newtheorem{corollary}[theorem]{Corollary}

\theoremstyle{definition}

\newtheorem{definition}[theorem]{Definition}

\newtheorem{remark}[theorem]{Remark}

\newtheorem{example}[theorem]{Example}

\allowdisplaybreaks

\newcommand{\equa}{\begin{eqnarray*}}
\newcommand{\tion}{\end{eqnarray*}}
\newcommand{\equal}{\begin{eqnarray}}
\newcommand{\tionl}{\end{eqnarray}}
\newcommand{\m}{\mathbbm{m}}
\newcommand{\DD}{{\mathbb{D}_{1,2}}}
\newcommand{\les}{\hspace{-2em}}
\newcommand{\vare}{\varepsilon}

\newcommand{\al}{\alpha}

\newcommand{\om}{\omega}
\newcommand{\Om}{\Omega}
\newcommand{\ftn}{\mathcal{F}}

\makeatletter
\def\timenow{\@tempcnta\time
\@tempcntb\@tempcnta
\divide\@tempcntb60
\ifnum10>\@tempcntb0\fi\number\@tempcntb
:\multiply\@tempcntb60
\advance\@tempcnta-\@tempcntb
\ifnum10>\@tempcnta0\fi\number\@tempcnta}
\makeatother

\title{Malliavin derivative of random functions and applications to L\'evy driven BSDEs }

\author{Christel Geiss$^1$   \hspace{1.3em}   Alexander Steinicke$^2$\\ 
}
\date{}
\begin{document}

\maketitle
\begin{abstract}

We   consider measurable  $F: \Omega \times \mathbb{R}^d \to \mathbb{R}$ where  for any $x$ the random variable $F(\cdot, x)$ belongs  to the  Malliavin Sobolev space $\mathbb{D}_{1,2}$ (with respect to a L\'evy process)  and 
provide  sufficient conditions on $F$ and  $G_1,\ldots,G_d \in \mathbb{D}_{1,2}$ such that  $F(\cdot, G_1,\ldots,G_d) \in \mathbb{D}_{1,2}.$ 

The  above result  is applied to  show Malliavin differentiability of solutions to BSDEs (backward stochastic differential equations) driven by L\'evy noise where  the generator 
is  given by a progressively measurable function $f(\omega,t,y,z).$   

\end{abstract}

\vspace{1em} 
{\noindent \textit{Keywords:} Malliavin calculus for L\'evy processes;  L\'evy driven BSDEs.
}
{\noindent
\footnotetext[1]{Department of Mathematics and Statistics, University of Jyv\"askyl\"a, Finland. \\ \hspace*{1.5em}
 {\tt christel.geiss{\rm@}jyu.fi}}
\footnotetext[2]{Department of Mathematics, University of Innsbruck, Austria. \\ \hspace*{1.5em}  {\tt alexander.steinicke{\rm@}uibk.ac.at}}}


\section{Introduction}\label{1}
Backward stochastic differential equations (BSDEs)  have been studied with growing interest  and from various perspectives. They appear in stochastic control theory,
as Feynman-Kac representation of second order semilinear PDEs, and have many applications in Finance and Insurance (see, for instance,  El Karoui et al.  \cite{ElKarouiPengQuenez},
the survey paper from Bouchard et al.   \cite{BouchardElieTouzi} or  Delong \cite{Delong}, and the references therein).

 Pardoux and Peng  have considered in  \cite{PP1} and \cite{PP2}   Forward Backward SDEs (FBSDEs) of the form
\equa
X_s &=& x + \int_t^s a(X_r)  dr + \int_t^s b(X_r)dW_r   \\
Y_s  &=& g(X_T) + \int_s^T f(X_r,Y_r,Z_r) dr +  \int_s^T Z_rdW_r, \quad t \le s\le T,
\tion
where $W$ denotes the Brownian motion. Under suitable smoothness and boundedness conditions on the coefficients they have shown that the two-parameter
process $\D_{\theta} Y_s $ is a.s. continuous in $s \in [\theta, T]$ and, moreover,   $\{ \D_{\theta} Y_{\theta} := \lim_{s \downarrow \theta} \D_{\theta} Y_s :   \theta \in [t,T]\} $
is a version of the process $\{Z_s: s \in [t, T]\}.$ In this way, using the relation
\[
    Y_s= \E \left [ g(X_T) + \int_s^T f(X_r,Y_r,Z_r) dr  \middle | \ftn^W_s \right]
\]
it is possible to represent $Z$ (with the right interpretation) as

\[
\left (\D_s \E\left  [ g(X_T) + \int_s^T f(X_r,Y_r,Z_r) dr \middle | \ftn^W_s \right ]   \right )_{s \in [t, T]}.
\]
These representations  turned out to be useful in regularity estimates for $Y$ and $Z$ which  play an important role for estimates  of convergence rates of
time-discretiza\-tions (see, for example, \cite{BouchardElie}, \cite{BouchardTouzi}, \cite{BouchardElieTouzi}, \cite{ImkellerDelong}).  \\
El Karoui et al. generalized in \cite{ElKarouiPengQuenez}   this result to  a class of progressively measurable  generators  $(\om,t) \mapsto f(\om,t,y,z).$  
 Also in the Brownian setting, Ankirchner et al. \cite{AnImRe}  and  Mastrolia et al. \cite{MaPoRe} extended the result to generators of BSDEs with quadratic growth. 

On the canonical L\'evy space, Malliavin differentiability of BSDEs  with jumps has been considered by Delong  in  \cite{Delong} and 
by Delong and Imkeller for  delayed BSDEs  in \cite{ImkellerDelong}.  \\
\smallskip
In this paper, we first consider a measurable function $F: \Om \times \R^d \to \R$ where  $F(\cdot, x)$ belongs  to the  Malliavin Sobolev space $\DD$ for  any $x \in \R^d.$
We ask for sufficient conditions on $F$ and  $G_1,\ldots,G_d \in \DD$ such that  $F(\cdot, G_1,\ldots,G_d) \in \DD.$
Our aim was to find very general conditions such that the result is also applicable for BSDEs with non-Lipschitz generators.
As we  work  in the L\'evy setting, the results hold of course especially for the Brownian case. In this respect, we could generalize
the conditions given in  \cite[Theorem 5.3]{ElKarouiPengQuenez} by not imposing the finiteness of fourth moments on the generator and the terminal condition 
(see Theorem \ref{diffthm} below). Moreover, we  provide a rigorous proof of the extended chain rule
for the Malliavin derivative of $F(\cdot, G_1,\ldots,G_d)$ in the Brownian case (see Theorem \ref{Mall-diff-thm}).
Such  a chain rule was already used in \cite{ElKarouiPengQuenez}.
Compared with  \cite{Delong}  or   \cite{ImkellerDelong}, we do  not require  a canonical L\'evy space to state Malliavin differentiability of BSDEs  (Theorem \ref{diffthm}).
\bigskip

The paper is organized as follows: Section 2 contains the setting and a collection of used notation. \\
Section 3 starts with the definition of the Malliavin derivative in the
L\'evy setting.   The Malliavin calculus based on chaos expansions in the L\'evy case has been treated in various papers, e.g. by L\o kka \cite{Loekka2}, Lee and Shih \cite{LeeShih}, Di Nunno et al. \cite{dinun}.In our paper, we recall a method used in  \cite{Steinicke} which is related to Picard's difference operator approach  \cite{picard}.  It allows to compute the Malliavin derivative $\D_{t,x}$ for $x\neq 0$ without knowing the chaos expansion
and without imposing the condition that the underlying probability space is  specified, e.g. as the canonical L\'evy space from \cite{suv2} or the probability space of Section 4 in \cite{Loekka2}.  Based on the fact that  $\D_{t,x}$ for $x\neq 0$ and $\D_{t,0}$ are of different nature we solve the question about the Malliavin  differentiability of $F(\cdot, G_1,\ldots,$ $ G_d) \in \DD$ in two steps:  In Subsection 3.3.1 we treat the question concerning $\D_{t,x}, x\neq 0$, while Subsection 3.3.2 contains the case $\D_{t,0}$. In the latter, we use the result from \cite{Sugita}  that for the Brownian motion the Malliavin Sobolev spaces
$\mathbb{D}^W_{1,p}(E)$ with $p >1$ ($E$ denotes a separable Hilbert space) coincide with the Kusuoka-Stroock Sobolev spaces which are defined  using the concept of ray absolute continuity and
stochastic Gateaux differentiability.   \\
In Section 4 we formulate the conditions on the BSDE such that it is Malliavin differentiable, present the proof and give an example.

\section{Setting}

Let $X=\left(X_t\right)_{t\in{[0,T]}}$ be a c\`adl\`ag L\'evy process on a complete probability space $(\Omega,\mathcal{F},\mathbb{P})$
with L\'evy measure $\nu$. We will denote the augmented natural filtration of $X$ by
$\left({\mathcal{F}_t}\right)_{t\in{[0,T]}}$ and assume that $\mathcal{F}=\mathcal{F}_T.$ \\
 \bigskip
 The L\'evy-It\^o decomposition of a L\'evy process $X$ can be written as
\begin{equation}\label{LevyIto}
X_t = \gamma t + \sigma W_t   +  \int_{{]0,t]}\times \{ |x|\le1\}} x\tilde{N}(ds,dx) +  \int_{{]0,t]}\times \{ |x|> 1\}} x  N(ds,dx),
\end{equation}
where $\sigma\geq 0$, $W$ is a Brownian motion and $N$ ($\tilde N$) is the (compensated) Poisson random measure corresponding to $X$.\\
The process $$\left(\int_{{]0,t]}\times \{ |x|\le1\}} x\tilde{N}(ds,dx) +  \int_{{]0,t]}\times \{ |x|> 1\}} x  N(ds,dx)\right)$$
is the jump part of $X$ and will be denoted by $J$. Note that the $\mathbb{P}$-augmented filtrations $(\ftn^W_t)_{t\in{[0,T]}}$ resp. $(\ftn^J_t)_{t\in{[0,T]}}$ generated by the processes $W$ resp. $J$ satisfy $$\ftn^W_t\vee\ftn^J_t=\ftn_t,$$ (see \cite[Lemma 3.1]{suv2}) thus spanning the original filtration generated by $X$ again.
Throughout the paper we will use the notation $X(\omega)=\left(X_t(\omega)\right)_{t\in{[0,T]}}$  for  sample trajectories. Let $\Delta X$ given by
$\Delta X_t:=X_t-\lim_{s\nearrow t}X_s$ denote the process of the jumps of $X$.\\
\bigskip
 Let
\[\mu(dx):=\sigma^2\delta_0(dx)+\nu(dx)\]
and
\equa
\m(dt,dx) :=(\lambda\otimes\mu) (dt,dx)
\tion
where $\lambda$ denotes the Lebesgue measure. We define the independent random measure  (in the sense of \cite[p. 256]{ito}) $M$ by
\equal  \label{measureM}
   M(dt,dx):=\sigma dW_t\delta_0(dx) +\tilde N(dt,dx)
\tionl
 on sets $B \!\in \!\mathcal{B}([0,T]\times\R)$ with  $\m(B) < \infty$.
It holds $\E M(B)^2 = \m(B).$

 In \cite{suv2}, Sol\'e et al.~consider  the independent random measure $\sigma dW_t\delta_0(dx)$ $+$ $x\tilde N(dt,dx).$ Here, in order to match the notation used for BSDEs, we work with the {\it equivalent} approach where the
Poisson random measure is not multiplied with $x$.\smallskip \\
We close this section with notation
for  c\`adl\`ag  processes on the path space and for BSDEs. \\ \bigskip

{\bf Notation:  Skorohod space}
\begin{itemize}
\item With $D{[0,T]}$ we denote the Skorohod space of c\`adl\`ag functions on the interval ${[0,T]}$ equipped with the Skorohod topology. The $\sigma$-algebra
$\mathcal{B}(D{[0,T]})$  is the Borel $\sigma$-algebra  i.e.~it is generated by the open sets of $D{[0,T]}.$ It coincides with the $\sigma$-algebra generated by the family of coordinate
projections $\left(p_t\colon D{[0,T]}\to \R,\ \tx  \mapsto \tx(t),\ t\geq 0\right)$  (see Theorem 12.5 of  \cite{Billing} for instance).
\item For a measurable mapping $\mathrm{Y}\colon\Omega\to D{[0,T]},\omega \mapsto \mathrm{Y}(\omega)$, the probability
measure $\mathbb{P}_\mathrm{Y}$ on $\left(D{[0,T]},\mathcal{B}\left(D{[0,T]}\right)\right)$ denotes the image measure of $\mathbb{P}$ under $\mathrm{Y}$.
\item For a fixed $t\in{[0,T]}$ the
notation
\equal \label{cut-off-in-t}
\tx^t(s):=\tx(t\wedge s),\text{ for all } s\in{[0,T]}
\tionl  induces the natural identification $$D{[0,t]}=\left\{\tx\in D{[0,T]} : \tx^t=\tx \right\}.$$ By this identification
we define a filtration on this space by
\equal \label{filtrationG-t}
\mathcal{G}_t=\sigma\left(\mathcal{B}\left(D{[0,t]}\right)\cup \mathcal{N}_X{[0,T]}\right), \quad 0\leq t\leq T,
\tionl where $\mathcal{N}_X{[0,T]}$ denotes the
null sets of $\mathcal{B}\left(D{[0,T]}\right)$ with respect to the  image measure $\mathbb{P}_X$  of the L\'evy process $X$. For more details on $D{[0,T]}$, see \cite{Billing} and \cite[Section 4]{delzeith}.
\end{itemize}
{\bf Notation for  BSDEs}
\begin{itemize}
\item For  $1\le p \le \infty$ let  $\mathcal{S}_p$ denote the  space of all $(\mathcal{F}_t)$-progressively measurable and c\`adl\`ag processes  $Y\colon\Omega\times{[0,T]} \rightarrow \R$ such that
\equa
\left\|Y\right\|_{\mathcal{S}_p}:=\|\sup_{0\leq t\leq T} \left|Y_{t}\right| \|_{\mathrm{L}_p} <\infty.
\tion

\item We define $\Ltwo(W) $ as the space of all $(\mathcal{F}_t)$-progressively measurable processes $Z\colon \Omega\times{[0,T]}\rightarrow \R$  such that
\equa
\left\|Z\right\|_{\Ltwo(W) }^2:=\E\int_0^T\left|Z_s\right|^2 ds<\infty.
\tion

\item Let $\R_0:= \R\!\setminus\!\{0\}$. We define $\Ltwo(\tilde N)$ as the space of all random fields $U\colon \Omega\times{[0,T]}\times{\R_0}\rightarrow \R$ which are measurable with respect to
$\mathcal{P}\otimes\mathcal{B}(\R_0)$ (where $\mathcal{P}$ denotes the predictable $\sigma$-algebra on $\Omega\times[0,T]$ generated
by the left-continuous $(\mathcal{F}_t)$-adapted processes) such that
\equa
\left\|U\right\|_{\Ltwo(\tilde N) }^2:=\E\int_{{[0,T]}\times{\R_0}}\left|U_s(x)\right|^2 ds \nu(dx)<\infty.
\tion
\item We define $\Ltwo(M)$ by $\Ltwo(M):=\Ltwo(W)\oplus\Ltwo(\tilde N)$ which is the space of all random fields $\underline{Z}\colon \Omega\times{[0,T]}\times\R\rightarrow \R$ which are measurable with respect to
$\mathcal{P}\otimes\mathcal{B}(\R)$  such that
\equa
\left\| \underline{Z}
\right\|_{\Ltwo(M) }^2:=\E\int_{[0,T]\times\R}\left|\underline{Z}_{s,x}\right|^2 \m(ds,dx)<\infty.
\tion

\item $\Ltwo(\nu):= \Ltwo(\R_0, \mathcal{B}(\R_0), \nu).$

\item $|\cdot|$ denotes a norm in $\R^n.$

\item For later use we recall the notion of the predictable projection of a stochastic process depending on parameters.

According to \cite[Proposition 3]{StrickerYor}  (see also  \cite[Proposition 3]{Meyer} or \cite[Lemma 2.2]{Ankirch}) for any $z\in\Ltwo(\PP\otimes\m):=
\Ltwo(\Omega\times {[0,T]}\times\R,\mathcal{F}_T\otimes\mathcal{B} ( [0,T]\times\R),
\mathbb{P} \otimes\m )$  there exists a process
\[^pz \in \mathrm{L}_2\left(\Omega\times {[0,T]}\times\R,\mathcal{P}\otimes\mathcal{B}(\R), \mathbb{P}\otimes\m\right)\]
such that for any fixed  $x\in\R$ the function
$ (^pz)_{\cdot,x}$ is a version of the predictable projection (in the classical sense, see e.g. \cite[Definition 2.1]{Ankirch}) of $ z_{ \cdot,x}.$ In the following we will always use this result to get predictable projections which are measurable w.r.t. a parameter. Again, we call $^pz$ the predictable projection of $z$.
\end{itemize}


\section{Malliavin calculus}
\subsection{Definition of \texorpdfstring{$\DD$}{DD} using chaos expansions}

The  random measure $M$   defined in  \eqref{measureM}  allows  to introduce the Malliavin derivative defined via chaos expansions (see, for example, \cite{suv}) as follows:
Any  $\xi \in \mathrm{L}_2:=\mathrm{L}_2(\Om,\ftn,\mathbb{P})$  has a unique chaos expansion (see \cite[Theorem 2]{ito})
\begin{equation*}
\xi=\sum_{n=0}^\infty I_n(\tilde f_n)
\end{equation*}
and it holds
\equa
    \E \xi^2:=   \|\xi\|^2_{\Ltwo}= \sum_{n=0}^\infty n!\left\|\tilde f_n\right\|_{\mathrm{L}^n_2}^2
\tion
where the $\tilde f_n\in  \widetilde{\mathrm{L}}_2^n:=\widetilde{\mathrm{L}}_2\left({\left({[0,T]}\times\R\right)}^n,\m^{\otimes n}\right),$ the subspace of symmetric functions
from $\mathrm{L}^n_2:=\mathrm{L}_2\left({\left({[0,T]}\times\R\right)}^n,\m^{\otimes n}\right),$
 and $I_n$ denotes the $n$-th multiple integral with respect to $M$ from \eqref{measureM}. The multiple integrals with respect to $M$ can be defined as follows:
If $n=0$ set $\mathrm{L}_2^0:=\R$ and $I_0(f_0):=f_0$ for $f_0\in\R.$ For $n\ge 1$ we start with a simple function   $f_n\in  \mathrm{L}_2^n$ given by
 $$f_n\left((t_1,x_1),\dotsc,(t_n,x_n)\right)=\sum_{k=1}^m a_k \prod_{i=1}^n\one_{B_i^k}(t_i,x_i),$$
where the sets $B_i^k \in \mathcal{B}({[0,T]}\times\R)  $  for $k=1,\ldots,m, i=1,\dotsc,n$ are disjoint for fixed $k$, and $\m(B_i^k)< \infty$ for all $i$ and $k.$ Then $$I_n(f_n):=\sum_{k=1}^m a_k \prod_{i=1}^n M(B_i^k).$$
By denseness of these simple functions in $ \mathrm{L}_2^n$ and by linearity and continuity of $I_n$, one extends the domain of the $n$-fold multiple stochastic integral $I_n$ to become a mapping $I_n\colon  \mathrm{L}_2^n\to  \mathrm{L}_2.$ It holds $I_n(f_n) =I_n(\tilde{f}_n)$ where $\tilde{f}_n$
denotes the symmetrization of $f_n$ w.r.t.~the $n$ pairs of variables in $[0,T]\times \R.$
For $f_n\in   \mathrm{L}_2^n$ and   $g_m\in  \mathrm{L}_2^m$ we have
$$\E I_n(f_n)I_m(g_m)=\begin{cases} n! \int_{([0,T]\times\R)^n}\tilde f_n \tilde g_n d\m^{\otimes n},\quad &n=m,\\ 0,\quad &n\neq m.\end{cases}$$

The space  $\DD$ consists of all  random variables $\xi \in \Ltwo$  such that
\equa
   \|\xi\|^2_{\DD}:= \sum_{n=0}^\infty (n+1)!\left\|\tilde f_n\right\|_{\mathrm{L}^n_2}^2<\infty.
\tion
The Malliavin derivative is defined for $\xi \in \DD$ by
\begin{equation*}
\D_{t,x}\xi:=\sum_{n=1}^\infty nI_{n-1}\left(\tilde f_n\left((t,x),\ \cdot\ \right)\right),
\end{equation*}
for $\mathbb{P}\otimes\m$-a.a. $(\omega,t,x)\in\Omega\times{[0,T]}\times\R$. Thus $\D \xi  \in  \Ltwo(\mathbb{P}\otimes\m )$.  \bigskip \\
We also consider
\equal \label{D-0}
\mathbb{D}_{1,2}^0&:= & \bigg \{\xi=\sum_{n=0}^\infty I_n(\tilde f_n) \in \Ltwo\colon \tilde f_n \in    \widetilde{\mathrm{L}}_2^n, n\in \N,  \non \\
   &&  \quad \quad \quad  \sum_{n=1}^\infty  (n+1)! \int_0^T  \|\tilde f_n((t,0),\cdot) \|_{\mathrm{L}^{n-1}_2}^2 dt < \infty  \bigg \}
\tionl
and
\equa
\mathbb{D}_{1,2}^{\R_0}&:= & \bigg \{\xi=\sum_{n=0}^\infty I_n(\tilde f_n) \in \Ltwo\colon \tilde f_n \in    \widetilde{\mathrm{L}}_2^n, n\in \N, \\
   &&  \quad \quad \quad  \sum_{n=1}^\infty  (n+1)! \int_{[0,T]\times \R_0}  \|\tilde f_n((t,x),\cdot) \|_{\mathrm{L}^{n-1}_2}^2 \m( dt,dx) < \infty  \bigg \}.
\tion

 \equal \label{D-intersection}
 \text{ If } \sigma >0  \text{ and  }   \nu \neq 0    \text{\,\, it holds  \,\, }  
\DD = \mathbb{D}_{1,2}^0 \cap \mathbb{D}_{1,2}^{\R_0}.
\tionl 
\subsection{From canonical to general probability spaces}
\bigskip
Sol\'e et al. introduced in \cite{suv2} the canonical  L\'evy space and proved that for $x\neq 0$ the  Malliavin derivative  $\D_{r,x} \xi$  (defined via chaos expansions)  equals in this space an increment quotient.   We will  discuss here how to transfer results   about random variables  from the canonical  L\'evy space to any  general probability space carrying a  L\'evy  process
provided that the regarded $\sigma$ -algebra is the completion of the one generated by the  L\'evy  process.\smallskip \\
This technique is needed, since key theorems of this section, like Theorem \ref{Mall-diff-thm}, will be proven on specific probability spaces. However, the formulation of its assertion is possible also on general probability spaces. The validity of the assertion is then guaranteed by the transfer technique  given in Theorem \ref{kernel-unique}. Hence, in Section 4, where we apply this section's theorems to BSDEs, we are not restricted to certain specific probability spaces.\smallskip

Assume $\left(\Omega_1,\mathcal{F}_1,\mathbb{P}_1\right), \left(\Omega_2,\mathcal{F}_2,\mathbb{P}_2\right)$ to be complete probability spaces with c\`adl\`ag L\'evy processes $X^i=(X^i_t)_{t \in [0,T]}$, $X^i_t\colon \Omega_i\to \R$, such that $X^i$ corresponds to a given L\'evy triplet $(\gamma,\sigma,\nu)$
for $i=1,2$. Furthermore, assume that $\mathcal{F}_i$ is the completion of the $\sigma$-algebra generated by $X^i$.
For the processes $X^1, X^2$, we get the associated independent random measures $M^1$ and  $M^2$ like in \eqref{measureM},  and the  families of multiple stochastic integrals $$\left(I^1_n(f_n)\right)_{n\in\N}, \left(I^2_n(f_n)\right)_{n\in\N},$$ respectively. The  following assertion  is taken from \cite[Corollary 4.2]{Steinicke}, where it is formulated for L\'evy  processes with paths in $D[0,\infty[$.

\begin{theorem}[]
\label{kernel-unique}
Let $(E,\mathcal{E},\rho)$ be a $\sigma$-finite measure space and let $$C^1\in\Ltwo\left(\Omega_1\times E,\mathcal{F}_1\otimes\mathcal{E},\mathbb{P}_1\otimes\rho\right),$$ $$C^2\in\Ltwo\left(\Omega_2\times E,\mathcal{F}_2\otimes \mathcal{E},\mathbb{P}_2\otimes\rho\right)$$
and suppose that these random fields have chaos decompositions
\begin{equation*}
C^1=\sum_{n=0}^\infty I^1_n(f_n),\ \mathbb{P}_1\otimes\rho\text{-a.e.},\quad C^2=\sum_{n=0}^\infty I^2_n(g_n),\ \mathbb{P}_2\otimes\rho\text{-a.e.}
\end{equation*}
for $f_n,g_n$ being functions in $\Ltwo(E,\mathcal{E},\rho)\hat\otimes\Ltwo^{n}$ which are symmetric in the last $n$ variables, where '$\hat \otimes$' denotes the Hilbert space tensor product.

Assume that for $\rho$-almost all $e\in E$ there are functionals $$F_{e}\colon D\left({[0,T]}\right)\to\R$$
such that $C^i(e)=F_{e}\left((X^i_t)_{t \in [0,T]}\right)$, $\mathbb{P}_i$-a.s. for $i=1,2$. Then for all $n\in\N$ it holds $f_n=g_n$, $\rho\otimes\m^{\otimes n}$-a.e.
\end{theorem}

Roughly speaking, if we have the same functionals  $F_{e}$ acting on both L\'evy processes $X^i$ defined on the probability spaces $(\Omega_i,\mathcal{F}_i,\mathbb{P}_i)$ for $i=1,2$ then  the deterministic kernels of  their  chaos expansions coincide. \bigskip \\
The Factorization lemma (see, for instance, \cite[ Section II.11]{bauer}) implies that for any $\xi \in  \Ltwo$ there exists a measurable functional $g_\xi \colon D([0,T]) \to \R$ such that
 \[ \xi(\omega) =g_\xi\left(\left(X_t(\omega)\right)_{0\leq t\leq T}\right) =g_\xi(X(\omega))\] for a.a. $\omega \in \Omega.$

  The following characterization that $g_\xi(X) \in  \mathbb{D}_{1,2} ^{\R_0}$ is a consequence from  Al\` os, Le\'on and Vives \cite[Corollary 2.3. and Lemma  2.1]{ Alos}  (this results hold true for a general L\'evy measure since  the  square integrability of the L\'evy process stated at the beginning of \cite{ Alos}  is in fact only used from   \cite[Section 2.4]{ Alos} on)  and Theorem \ref{kernel-unique}.  For details see the proof in  \cite[Theorem 5.1]{Steinicke}.

\begin{lemma} \label{functionallem}
If  $g_\xi(X) \in    \Ltwo$
then $g_\xi(X) \in  \mathbb{D}_{1,2} ^{\R_0}$  $\iff$
\begin{equation}\label{fctder}
 g_\xi(X+x\one_{[t,T]})-g_\xi(X) \in \Ltwo(\PP \otimes\m)
\end{equation}
and it holds then
   for    $x\neq0$  $\PP \otimes\m$-a.e.
\begin{equation}\label{xieq}
\D_{t,x} \xi=  g_\xi(X+x\one_{[t,T]})-g_\xi(X) .
\end{equation}
\end{lemma}
Compared to the approach of \cite{suv2} which uses the random measure $\sigma dW_t\delta_0(dx) +x\tilde N(dt,dx)$, here the according Malliavin derivative for $x\neq 0$ and $M$ from \eqref{measureM} is just a {\it difference} instead of the  difference quotient from \cite{suv2}.

 Applied on  $g_\xi(X(\omega))$ this gives in the canonical space
    \equa
      g_\xi(X(\omega_{r,x}))  -g_\xi(X(\omega)) =   g_\xi(X(\om)+x\one_{[r,T]})-g_\xi(X(\om))
  \tion
  for $\PP \otimes\m$ a.e. $(\om,r,x).$

In the situation of the previous lemma, one may ask whether properties of $g_\xi(X)$ that hold $\PP$-a.s. are preserved $\PP\otimes\m$-a.e. for
$g_\xi(X+x\one_{[t,T]})$. The positive answer is given  by the following result (the proof can be found in the appendix).
\begin{lemma}\label{sprungaddthm}
Let $\Lambda\in \mathcal{G}_T$ be a set with  $\mathbb{P}\left(\left\{X\in \Lambda\right\}\right)=0$. Then
$$\mathbb{P} \otimes\m\left(\left\{(\omega,r,v)\in \Omega\times{[0,T]}\times\R_0:X(\omega)+v\one_{[r,T]}\in \Lambda\right\}\right)=0.$$
\end{lemma}

\bigskip
\begin{corollary}  {\color{white}.} \label{kept-almost-sure}
\begin{enumerate}[(i)]
\item  \label{corA}
Let $f\colon D{[0,T]}\times\R\to\R$ be a measurable mapping such that $\PP$-a.s. $y\mapsto f(X(\omega),y)$ is a Lipschitz function with Lipschitz constant $L$ independent from $\omega\in\Omega$. Then the set
$$\Lambda:=\left\{ {\tt x}\in D{[0,T]}: y\mapsto f({\tt x},y) \text{ is not Lipschitz in } y \text{ with constant }L \right\}$$
satisfies $\PP(X\in\Lambda)=0$. Lemma \ref{sprungaddthm} implies that also $$y\mapsto f(X(\omega)+v\one_{[r,T]},y)$$ is a Lipschitz function with constant $L$ for $\PP\otimes\m$-a.e. $(\omega,r,v)\in\Omega\times{[0,T]}\times\R_0$.
\item
Let $\xi=g_\xi(X)\in L_{\infty}(\Omega)$. By the same reasoning as in \eqref{corA}, it follows from Lemma \ref{sprungaddthm} that $\PP\otimes\m$-a.e. the random element $g_\xi(X+v\one_{[r,T]})$ is bounded.
\end{enumerate}
\end{corollary}
Note that the boundedness of $g_\xi(X+v\one_{[r,T]})$ implies boundedness of the difference in \eqref{xieq},
$$g_\xi(X+v\one_{[r,T]})-g_\xi(X),$$
which -- in case of $\Ltwo$-integrability w.r.t.~$\PP\otimes\m$ -- equals the Malliavin derivative for $v\neq 0$.
\bigskip

\subsection{Malliavin calculus for random functions}

 We want to address the following problem: Let
\[
   F: \Om \times \R^d \to \R
\]
be jointly measurable, for any $y \in \R^d$ we assume $F(\cdot, y) \in \DD,$ and  for a.a. $\om\in\Omega$
let $F(\om, \cdot)\in \mathcal{C}^1(\R^d).$ If $G_1,...,G_d \in \DD,$ under which assumption do we get
\[
  F(\cdot, G_1,...,G_d) \in \DD ?
\]

We will treat this question in two steps: First we will find conditions on  $F$  and  $G=(G_1,...,G_d)$  such that
  \begin{itemize}
  \item   $ F(\cdot, G)     \in   \mathbb{D}_{1,2}^{\R_0}   $
  \item $ F(\cdot, G) \in    \mathbb{D}_{1,2}^0 $
  \end{itemize}
  separately and then  use relation \eqref{D-intersection}. \bigskip
\subsubsection{The case \texorpdfstring{ $ F(\cdot, G) \in \mathbb{D}_{1,2}^{\R_0} $}{FinDR}}
\begin{lemma} \label{functionallem1}
 Assume that $F(\cdot, y) \in \mathbb{D}_{1,2}^{\R_0}\,$ for all $y \in \R^d,$   $F(\cdot, G)  \in  \Ltwo,$  and  $G_1,...,G_d \in \mathbb{D}_{1,2}^{\R_0}.$
 Let $F(\om, \cdot)\in \mathcal{C}(\R^d)$ $\PP$-a.s. and let $F$ be represented by the functional $g_F(X,\cdot)$. Then $F(\cdot, G)     \in   \mathbb{D}_{1,2}^{\R_0} $     $\iff$
\equal \label{jump-derivative}    ( \D_{t,x} F)(\cdot,   G) + \hspace{-1.5em} &&g_F(X + x\one_{[t,T]},   G+  \D_{t,x}G)   -g_F(X+ x\one_{[t,T]}, G) \non \\
&& \quad\quad\quad
 \in \Ltwo(\Om \times [0,T] \times  \R_0, \PP \otimes\m) .
\tionl
\end{lemma}
\begin{proof} By the expression $( \D_{t,x} F)(\cdot, G)$ we mean that
we insert the $\Ltwo$-vector $(G_1,\dotsc,G_d)$ into the $y$-variable of $\D_{t,x}F(\cdot,y)$.
Furthermore, since by Lemma \ref{sprungaddthm}, expression  $g_F(X(\omega)+ x\one_{[t,T]}, y)$ is continuous in $y$ for $\PP\otimes\m$-a.e.
$(\omega,t,x)\in\Omega\times[0,T]\times\R_0$, taking equivalence classes of
\begin{align*} g_F(X(\omega)+ x\one_{[t,T]},y)\mid_{y=(G_1(\omega)+ \D_{t,x}G_1(\omega),\dotsc,G_d(\omega) +  \D_{t,x}G_d(\omega))} \end{align*}
for representatives $(G_1(\omega)+ \D_{t,x}G_1(\omega),\dotsc,G_d(\omega) +  \D_{t,x}G_d(\omega))$ leads to a well-defined
 $\mathrm{L}_0(\PP\otimes\m)$ object.

 For the sufficiency, one can use the same arguments as for   \cite[Theorem 5.2]{Steinicke}. There the proof is carried out only for
 $d=1$  but it is easy to see that the multidimensional case can be proved in the same way. \\
For the  necessity we consider $G_1, \dotsc, G_d$ as given by functionals $g_{G_1},$ $\dotsc,$ $g_{G_d}$ and conclude from Lemma \ref{functionallem} that
 \equa
 \D_{t,x} F(\cdot, y) = g_F(X+x\one_{[t,T]},y) - g_F(X,y).
 \tion
 Hence expression \eqref{jump-derivative} equals in fact
 \equa
 && \les g_F(X+x\one_{[t,T]},G_1+   \D_{t,x}G_1  ,...,G_d +   \D_{t,x}G_d) -
 g_F(X,G_1,...,G_d) \\
 &=&\! g_F(X+x\one_{[t,T]}, g_{G_1}\!(X+x\one_{[t,T]}),..., g_{G_d}\!(X+x\one_{[t,T]})) -
 g_F(X,G_1,...,G_d) \\
 &=& \!\D_{t,x}  F(X,G_1,...,G_d)
 \tion
 where we have used Lemma \ref{functionallem} again.
 \end{proof}

\subsubsection{ The case \texorpdfstring{ $F(\cdot, G) \in    \mathbb{D}_{1,2}^0$}{FinD}}
The L\'evy-It\^o decomposition implies that the Brownian part and the pure jump part of a L\'evy process are independent. Thus we may represent a copy of $X$ on the completion of  $(\Om^W\times\Om^J, \ftn^W\otimes\ftn^J, \PP^W\otimes\PP^J)$
as
\equa
   X_t(\om) = \gamma t + \sigma \om^{_W}_t  + J_t(\om^{_J}), \quad t \in [0,T],
\tion
where  $\om=(\om^{_W},  \om^{_J}).$
Here $(\Om^W, \ftn^W, \PP^W)$ denotes the completed canonical Wiener space i.e.  $\Om^W := \mathcal{C}_0[0,T]$ is the space of continuous functions starting in $0,$ and
$\F^W$ is the Borel $\sigma$-algebra completed with respect to the Wiener  measure $\PP^W.$ The space $(\Om^J, \ftn^J, \PP^J)$ is a probability space carrying the pure jump process $J$, where $\ftn^J$ is generated by $J$ and completed.\bigskip  \\
To work on the canonical space  $(\Om^W, \ftn^W, \PP^W)$ we continue with a short reminder on Gaussian Hilbert spaces and refer the reader for more information to Janson \cite{Janson}.
Consider the Gaussian Hilbert space $\h :=\big \{ \int_0^T h(s) dW_s: h \in \Ltwo[0,T] \big  \}.$ Because of It\^o's  isometry we may identify
$\h$ with $$\h_0:=\Ltwo[0,T].$$
The space
\[
\h_1 :=  \bigg \{ \int_0^\cdot h(s)ds: h \in \Ltwo[0,T] \bigg  \}
\]
with $\langle \int_0^\cdot h_1(s)ds,\int_0^\cdot h_2(s)ds\rangle_{\h_1} :=\! \int_0^T\! h_1(s) h_2(s)ds$ is the Cameron-Martin space.
For $h \in \h_0$ we have $g_h \in \h_1 $ with
\[  g_h(t):= \E \bigg (\int_0^T h(s) dW_s  \, W_t \bigg )= \int_0^t h(s)ds.
\]

The main idea to get sufficient conditions for $ F(\cdot,G) \in    \mathbb{D}_{1,2}^0$ consists in applying Theorem \ref{Sugita-thm} below.
We proceed with a collection of  definitions and some facts related to this theorem. \smallskip\\
In the sequel let $E$ be a separable Hilbert space.

\begin{definition} [\cite{Sugita}, \cite{Nualartbook}]
Let  $1 \le p < \infty $ and $\mathcal{S} \subseteq  \mathbb{D}_{1,p}(\PP^W) $ be a dense set of smooth random variables.
By $\mathbb{D}^W_{1,p}(E)$ we denote the completion of
\[
  \{ \xi= \sum_{k=1}^n G_k  H_k:   G_k \in  \mathcal{S}, H_k \in E   \}
\]
with respect to the norm
\[
   \| \xi\|_{1;E} := \left ( \E  \| \xi \|^p_E  + \E \left (\int_0^T \|\ D_t^W \xi  \|^2_E dt \right )^\frac{p}{2}  \right )^\frac{1}{p}
\]
where $ D_t^W \xi:=\sum_{k=1}^n (D_t^W G_k ) H_k.$
\end{definition}

Note that $\Ltwo(\Om^J, \ftn^J, \PP^J)$ is a separable Hilbert space, and that the space $\mathbb{D}^W_{1,2}(E)$
for $E:= \Ltwo(\Om^J, \ftn^J, \PP^J)$ can be identified with $\mathbb{D}_{1,2}^0$ defined in \eqref{D-0} (see  \cite{Alos}). This means we may
reformulate the question posed in the beginning of this section by asking for sufficient conditions such that
\[
F(\cdot, G) \in  \mathbb{D}^W_{1,2}(E).
\]
The answer will be  Theorem  \ref{Mall-diff-thm} at the end of this section. \\ \bigskip

Let $E_1$ and $E_2$ be separable Hilbert spaces. A bounded linear operator $A: E_1 \to E_2$ is is called
{\it Hilbert-Schmidt operator} if for some orthonormal basis $\{e_n\}$ in $E_1$ it holds
\[
    \|A\|_{HS(E_1,E_2)} := \bigg ( \sum_{n=1}^\infty \|Ae_n\|^2_{E_2} \bigg )^\half<\infty
\]
(see, for example, \cite{Bogachev}).
We will denote by $HS(\h_0,E)$ the space of  Hilbert-Schmidt operators between $\h_0$ and $E.$

\begin{definition} [\cite{Janson},\cite{Bogachev}]
With $\mathrm{L}_0(\PP^W; E)$ we denote the space of $E$-valued random variables, equipped with the topology
of convergence in probability.

  For   $\xi \in \mathrm{L}_0(\PP^W; E)$  and $h \in \h_0$ we define the {\it Cameron-Martin shift }  by
  \[\rho_h (\xi)(\om^{_W}) :=   \xi(\om^{_W} +g_h).\]
\end{definition}

One of the properties of the Cameron-Martin shift is the Cameron-Martin formula.  (For an integral of $E$-valued objects, we always use the Bochner integral.)
\begin{lemma} {\color{white}.} \label{cameron-martin}
\begin{enumerate}[(i)]
  \item (Cameron-Martin formula).    $\PP^W \sim  \PP^W \circ \rho_h^{-1} $  for  $h \in \h_0,$
  and the Radon-Nikodym derivative is given by
  \[ \frac{d \PP^W \circ \rho_h^{-1}}{d\PP^W }(\om^{_W}) = \exp \bigg\{-\half  \int_0^T h(t)^2 dt - \int_0^T h(t) dW_t \bigg \}.\]
  \item \label{ii} If $K \in  L_p(\PP^W; E)$ for some $p>1$ then for any $q \in {[1,p[}$
 \[ \bigg \| \int_0^T \rho_{sh} K ds \bigg \|_{L_q(\PP^W;E)} \le \int_0^T \exp \bigg \{ \frac{s^2}{2(p-q)} \|h\|^2_{\h_0} \bigg \} ds \,\, \|  K\|_{L_p(\PP^W;E)}.
  \]
  \item \label{iii} For $p\in{]0,\infty]}$, every $\xi \in \mathrm{L}_p(\PP^W)$ and for all $q\in{[0,p[}$, the map $$\h_0 \to  
\mathrm{L}_q(\PP^W): h \mapsto \rho_h(\xi)$$ is continuous. If $p=q=0$, continuity also holds.
  
 \end{enumerate}
\end{lemma}

\begin{proof}
(i)
 See Kuo \cite[Theorem 1.1]{Kuo}.  \\
(ii) Analogously to the proof of  Theorem 14.1 (vi) in Janson \cite{Janson}  for $1\le q<p$ we choose $r=\frac{p}{p-q}$  so that $\frac{1}{r}+\frac{q}{p}=1,$
  and by the Cameron-Martin formula and H\"olders inequality we get
 \equa
&& \bigg \| \int_0^T \rho_{sh} K ds \bigg \|_{L_q(\PP^W;E)} \\
 &\le& \int_0^T  \big (\E   \| \rho_{sh} K\|^q_E \big)^\frac{1}{q}   ds  \\
 &=& \int_0^T \bigg ( \E \exp \bigg \{s \int_0^T h(t)dW_t - \frac{s^2}{2} \|h\|^2_{\h_0} \bigg \} \| K\|^q_E  \bigg )^\frac{1}{q} ds \\
 &\le& \|  K\|_{L_p(\PP^W;E)}     \int_0^T \bigg ( \E \exp \bigg \{sr \int_0^T h(t)dW_t - \frac{s^2r}{2} \|h\|^2_{\h_0} \bigg \}  \bigg )^\frac{1}{rq} ds \\
 &=& \|  K\|_{L_p(\PP^W;E)}  \int_0^T \bigg (\exp \bigg \{   \frac{s^2(r^2-r)}{2} \|h\|^2_{\h_0} \bigg \}  \bigg )^\frac{1}{rq} ds \\
  &=& \|  K\|_{L_p(\PP^W;E)}   \int_0^T \exp \bigg \{ \frac{s^2}{2(p-q)} \|h\|^2_{\h_0} \bigg \} ds. \\
 \tion

(iii) This assertion is formulated for  real valued random variables in \cite[Theorem 14.1 (viii)]{Janson} but in  \cite[Remark 14.6]{Janson} it is stated
that it holds for random variables with values in a separable Banach space.
\end{proof}

\begin{definition} [\cite{Janson},\cite{Bogachev}]
\begin{enumerate}[(i)]
\item A random variable $\xi \in  \mathrm{L}_0(\PP^W; E)$  is {\it absolutely continuous along  $h \in \h_0 $} ($h$-a.c.)  if there exists a  random
variable $ \xi^h \in  \mathrm{L}_0(\PP^W; E)$     such that
$ \xi^h = \xi \,\,\, a.s.$ and for all $\om^{_W\!\!}  \in \Om^W$ the map  \[ u \mapsto \xi^h(\om^{_W\!\!}  + u \, g_h)  \]
 is absolutely continuous  on bounded intervals of $\R.$
 \item $\xi \in \mathrm{L}_0(\PP^W; E)$ is {\it ray absolutely continuous} (r.a.c.) if $\xi$ is $h$-a.c. for every $h \in \h_0.$
\item For  $\xi \in \mathrm{L}_0(\PP^W; E)$  and $h \in \h_0$ we say the {\it directional derivative  $\partial_h \xi \in \mathrm{L}_0(\Om^W; E)$ exists } if
\equa
\frac{ \rho_{u h} (\xi) - \xi}{u} \to^{\hspace*{-0.7em}\PP^W } \,    \partial_h \xi,  \quad u \to 0.
\tion
\item $\xi \in \mathrm{L}_0(\PP^W; E)$ is called  {\it stochastically G\^ateaux differentiable} (s.G.d.)    if  $\partial_h \xi$ exists
for every $h \in \h_0$ and there exists an  $HS(\h_0,E)$-valued random variable denoted by $\tilde \D \xi$ such that for every $h \in \h_0$
\equa
        \partial_h \xi=   \langle \tilde \D \xi ,h\rangle_{\h_0}, \quad   \PP^W\text{-} a.s.
\tion
\end{enumerate}
\end{definition}

According to Sugita \cite{Sugita},  the Malliavin Sobolev spaces $\mathbb{D}^W_{n,p}(E)$  for $n \in \N, 1<p< \infty$  and the  Kusuoka-Stroock Sobolev spaces
defined via the properties r.a.c. and s.G.d.~coincide. According to  Bogachev  \cite{Bogachev} this holds also for $p=1.$  Here we only use the assertion for $n=1$:

\begin{theorem} [  {\cite[Theorem 3.1]{Sugita}, \cite[Proposition 5.4.6 (iii)]{Bogachev}} ]
\label{Sugita-thm}
 Let $p \in [1, \infty[.$ Then
\equa
\mathbb{D}^W_{1,p}(E) \!= \!\{\xi \in \mathrm{L}_p(\PP^W\!; E)\colon \! \xi \text{ is  r.a.c., s.G.d. and } \tilde\D \xi \in \mathrm{L}_p(\PP^W\!\!; H\!S(\h_0;\!E)) \},
\tion
and  for  $\xi \in  \mathbb{D}^W_{1,p}(E) $  it holds  $\D^W \xi = \tilde \D \xi$ a.s.
\end{theorem}

We will also need the following result.

\begin{theorem} \label{Janson-thm}
For  $h \in \h_0$ and $\xi \in \mathrm{L}_0(\Om^W; E)$ it holds
\equa
\begin{array}{l}
\xi \text{ is }  \\
h\text{-a.c}
\end{array}
\!\!\!\!\!\iff \!\!\! \left \{ \begin{array}{l}
(i) \,\, \partial_h \xi \text{ exists} \\
(ii)  \forall u \in \R\colon\!  \rho_{uh}\xi(\om^{_W\!})\!-\xi(\om^{_W\!})  = \int_0^u \! \rho_{sh} (\partial_h \xi)(\om^{_W\!}) ds \\
\quad \,\,\PP^W\text{\!-a.s.,} \\
\quad\, \text{ where } \int_{-|u|}^{|u|}  \|\rho_{sh} (\partial_h \xi)(\om^{_W\!})\|_E \, ds < \infty \,\, \PP^W\text{-a.s.} \\
\quad \,\text{ and } (s, \om^{_W\!}) \mapsto  \rho_{sh} (\partial_h \xi)(\om^{_W\!}) \text{ denotes a jointly } \\
\quad\,    \text{measurable version.}
\end{array} \right.
\tion
\end{theorem}
\begin{proof}
For $E=\R$ this is Theorem 15.21 of \cite{Janson}. One can generalize the proof to $E$-valued random variables since by the Radon-Nikodym property of $E$
(see \cite[Corollary IV.1.4]{diestel}), the fundamental theorem of calculus holds for absolutely continuous functions if Bochner integrals are used.
\end{proof}
\bigskip

With the above preparations we are now  able to find sufficient conditions  for $F(\cdot, G) \in  \mathbb{D}^W_{1,2}(E).$

\begin{theorem} \label{Mall-diff-thm}
Assume that  $E= \Ltwo(\Om^J, \ftn^J, \PP^J)$ and $$(\Omega,\mathcal{F},\mathbb{P})=(\Om^W\times\Om^J, \ftn, \PP^W\otimes \PP^J),$$ where $\ftn$ is the
completion of $\ftn^W\otimes\ftn^J$.
Let
\[
   F: \Om  \times \R^d \to \R
\]
be jointly measurable and $G_1,...,G_d \in \mathbb{D}^W_{1,q}(E)$ for some  $q > 1$. Suppose that $p>1$ and
\begin{enumerate}[(i)]
\item \label{C-one}  $F(\om,\cdot) \in \mathcal{C}^1(\R^d)$ for  a.a.  $\om \in \Om$,
\item for all $y \in \R^d:$  $F(\cdot,y) \in  \mathbb{D}^W_{1,p}(E),$
\item  \label{C-three} for a function $\delta\colon{[0,\infty[}\to{[0,\infty[}$, continuous at zero, and for each $N\in \N,$ $\exists K_N \in \bigcup_{r >1} 
\mathrm{L}_r(\PP)$ such that  for a.a.~$\om$ it holds:
\equa
&&\forall y,\tilde y \in B_N(0):=\{x\in \R^d: |x|\le N\}: \\
 && \|  (\D^W F(\cdot, y))(\om) - (\D^W F(\cdot, \tilde y))(\om) \|_{\h_0} \le K_N(\om) \delta(|y - \tilde y|),
\tion
\item\label{C-four}      $(\D^W F)( \cdot,G_1, ...,G_d) \in \mathrm{L}_p(\PP^W; HS(\h_0,E)),$ \\
 $\frac{\partial}{\partial y_k }F( \cdot,G_1,...,G_d)\!\in \bigcup_{r>q'}\mathrm{L}_{r}(\PP^W;E)$  \ for\ $1\leq k\leq d$, \, $\frac{1}{q'}+ \frac{1}{q}=1,$ and
\equa \sum_{k=1}^d \frac{\partial}{\partial y_k }F( \cdot,G_1,...,G_d) \D^W G_k \in \mathrm{L}_p(\PP^W; HS(\h_0,E)). \tion
 \end{enumerate}
Then
\[ F( \cdot,G_1,...,G_d) \in \mathbb{D}^W_{1,p}(E)\]
and
\[
\D^W F( \cdot,G_1,...,G_d) = (\D^W F)( \cdot,G_1,...,G_d) +   \sum_{k=1}^d \frac{\partial}{\partial y_k }F( \cdot,G_1,...,G_d) \D^W G_k
\]
in $  \mathrm{L}_p(\PP^W; HS(\h_0,E)).$
\end{theorem}

\begin{remark} In Theorem \ref{Mall-diff-thm} it is possible to use also
 \equa \text{(iii)'} \quad
\forall \vare >0 \, \exists \delta_{\vare }(y) >0:  \forall \om \in \Om,  \forall \tilde y \in B_{\delta_{\vare}}(y):
\tion
\equa
  \|  (\D^W F(\cdot, y))(\om) - (\D^W F(\cdot, \tilde y))(\om) \|_{\h_0} \le \vare.
\tion
instead  of \eqref{C-three}. Neither of both assumptions implies the other one.
\end{remark}

\begin{proof} {\tt Step 1.}
We will  use  the characterization of  $\mathbb{D}^W_{1,p}(E)$  from Theorem \ref{Sugita-thm}.
In fact, we will prove for any $u\in \R$ and $h \in \h_0$ the relations
\begin{equation} \label{rac-aim}
\begin{split}
\rho_{uh} F(\om^{_W\!},G(\om^{_W\!})) -  F(\om^{_W\!},G(\om^{_W\!}))
=  \int_0^u  \rho_{sh}(\partial_h F) (\om^{_W\!},G(\om^{_W\!})) ds,\quad \PP^W\text{-a.s.},\\
(\partial_h F) (\cdot,G)
= \langle  (\D^W F)(\cdot,G) +  \nabla_y F(\cdot, G) \cdot  \D^W G,h\rangle_{\h_0}, \quad \PP\text{-a.s.}
\end{split}
\end{equation}
where the first equation is $E$-valued with $G=(G_1,\dotsc,G_d)$, and the second equation is scalar with $\nabla_y =(\frac{\partial}{\partial y_1},\dotsc, \frac{\partial}{\partial y_d})$. \bigskip\\
Since by assumption \eqref{C-four}
\[   (\D^W F)(\cdot,G) +  \nabla_y F(\cdot, G) \cdot  \D^W G   \in \mathrm{L}_p(\PP^W; HS(\h_0,E))\]
we infer that   $\int_{-|u|}^{|u|}  \|\rho_{sh} (\partial_h F) (\cdot,G)\|_E ds < \infty$, $\PP^W$-a.s and according to Theorem \ref{Janson-thm} it follows from the first line of \eqref{rac-aim} that $F(\cdot, G)$ is r.a.c.
From the second line of \eqref{rac-aim} we get that  $F(\cdot, G)$ is s.G.d.
 and
\[  \tilde \D F(\cdot,G) = (\D^W F)(\cdot,G) +  \nabla_y F(\cdot, G) \cdot  \D^W G   \]
in $ \mathrm{L}_p(\PP^W; HS(\h_0,E))$. Together with Theorem \ref{Sugita-thm} this would imply the assertion of the theorem. So it remains to show the relations in \eqref{rac-aim} which will be done in Steps 2 and 3.
\smallskip

{\tt Step 2.}
Since  $F(\cdot,y) \in \mathbb{D}^W_{1,p}(E)$   we have by  Theorem \ref{Sugita-thm}   that
 $F(\cdot,y)$ is r.a.c. and
 \[\D^W F(\cdot, y) =  \tilde \D F(\cdot, y)  \in   \mathrm{L}_p(\PP^W; HS(\h_0,E)). \]
 Hence by Theorem \ref{Janson-thm}
for each   $u \in \R$ and $h \in \h_0$  the $E$-{\it valued} equation
\equa
 F(\om^{_W\!\!} + u g_{h},y) - F(\om^{_W\!\!},y) = \int_0^u\rho_{sh}   \langle (  \D^W F(\cdot ,y))( \om^{_W\!\!})  ,h\rangle_{\h_0} ds,
\tion
holds for all $\om^{_W\!\!}$ up to an  exception set $C_y \in \ftn^W$ with $\PP^W(C_y)=0.$ Consequently, for each   $u \in \R$ and $h \in \h_0$ we have the
{\it real-valued} equation (where we use the notation  $\rho_{uh}(\om) = (\om^{_W} +ug_h, \om^{_J})$)
\equal \label{f-is-h-a-c}
 F(\rho_{uh}(\om), y) - F(\om,y) = \int_0^u \rho_{sh}   \langle (  \D^W F(\cdot ,y))( \om)  ,h\rangle_{\h_0} ds,
\tionl
for all $\om$  with the exception of a  set $\bar{C}_y \in \ftn$ with $\PP(\bar{C}_y)=0.$ Since the LHS is a.s.~continuous in $y$, we can find an exception set $\bar{C} \in \ftn$ with 
$\PP(\bar{C})=0,$ which is independent of $y,$  provided that we can show a.s.~continuity in $y$ of the RHS. To do this we estimate for $y, \tilde y  \in  B_N(0)$ the expression
  \equa
&&\les  \bigg | \int_0^u \rho_{sh}   \langle (  \D^W F(\cdot , \tilde y))( \om)  ,h\rangle_{\h_0} ds -\int_0^u \rho_{sh}   \langle (  \D^W F(\cdot ,y))( \om)  ,h\rangle_{\h_0} ds \bigg | \\
& \le &  \bigg | \int_0^u \rho_{sh}   \langle (  \D^W F(\cdot ,\tilde y))( \om) -   (  \D^W F(\cdot ,y))( \om)  ,h\rangle_{\h_0}    ds  \bigg | \\
& \le & \|h \|_{\h_0} \! \int_0^u \!\!  \|  (\D^W F(\cdot ,\tilde y))( \om^{_W\!\!}+s g_h,\om^{_J\!\!} ) -   (  \D^W F(\cdot ,y))( \om^{_W\!\!}+s g_h,\om^{_J\!\!}) \|_{\h_0}     ds  \\
& \le & \|h \|_{\h_0} \delta({|y-\tilde y|})    \int_0^u K_N(  \om^{_W\!\!}+s g_h, \om^{_J\!\!}   ) ds.
  \tion

Since by Lemma \ref{cameron-martin}    $\int_0^u \rho_{sh} K_N ds < \infty$, $\PP$-a.s., it follows that for a.a. $\om$  the RHS of  \eqref{f-is-h-a-c} is continuous in $y$.
Consequently,  on  $\Om\! \setminus \!\bar{C} \in \ftn$   relation \eqref{f-is-h-a-c}  is true  for all $y \in \R^d.$
Putting the terms to zero on $\bar{C}$, the right hand side of \eqref{f-is-h-a-c} is jointly measurable w.r.t.~$(\om,y)$.
We may replace  $y$ by $G(\om):=( G_1(\om),\ldots,G_d(\om))$
and get
\begin{align} \label{F-with-G-formula}
 F(\rho_{uh}(\om),G(\om)) - F(\om,G(\om)) = \int_0^u  \langle  (\D^W F)(\rho_{sh}(\om),G(\om)),h\rangle_{\h_0} ds, \,\,\PP\text{- a.s.}
\end{align}
So far the Cameron-Martin shift $\rho_{uh}$  acts only on the first variable of $F(\om,G(\om)).$ In the following step we derive the representation for
$ \rho_{uh}F(\om,G(\om)).$

{\tt Step 3.}
We show that $F(\cdot,G)$ is r.a.c. For this  we choose an interval $[0,t_1]$,  $t \in [0,t_1],$
let $ 0=s_0 < s_1 < ...< s_n =t_1$
and consider for $s^t_k := s_k \wedge t$  the expression
\begin{align} \label{sum-of-F}
 \les \rho_{th}F(\om,G(\om)) -  F(\om,G(\om))
= \sum_{k=1}^n\rho_{s_k^t h} F(\om,G(\om)) - \rho_{s_{k-1}^t h} F(\om,G(\om)). \non\\
\end{align}
For any $b:=s_k^t$ and $a:= s_{k-1}^t$ we derive from \eqref{F-with-G-formula} and the mean-value theorem that a.s.
\equa
&& \les\rho_{b h} F(\om,G(\om)) - \rho_{a h}F(\om,G(\om)) \\
& =& [F(\rho_{b h}(\om),G(\rho_{b h}(\om))) - F(\rho_{a h}(\om),G(\rho_{b h}(\om)))] \\
&&+[F(\rho_{a h}(\om),G(\rho_{b h}(\om))) - F(\rho_{a h}(\om),G(\rho_{a h}(\om)))] \\
&=& \int_a^b   \langle  (\D^W F)(\rho_{sh}(\om),G(\rho_{b h}(\om))),h\rangle_{\h_0} ds  \\
&& +  \nabla_y F\left(\rho_{a h}(\om), G(\rho_{a h}(\om)) + \theta  [ G(\rho_{b h}(\om)) - G(\rho_{a h}(\om))]\right) \\
&& \quad \quad   \quad \quad  \cdot [G(\rho_{b h}(\om)) - G(\rho_{a h}(\om))]
\tion
for some $\theta \in[0,1].$
We may write the last term because $F(\om, y)$ is $\mathcal{C}^1$ w.r.t. $y $.
Similarly to \eqref{f-is-h-a-c} ,  for  each $G_l \in \mathbb{D}^W_{1,q}(E),$ we have   for all $t \in \R$ and $h \in \h_0$ that
\equal \label{Y-h-a-c}
\rho_{th} G_l(\om) - G_l(\om) &=& \int_0^t      \langle ( \D^W G_l)(\rho_{sh}(\om)),h \rangle_{\h_0} ds,  \quad  \PP\text{-}a.s.  \non \\
  \text{ with }\quad   m(\om)&:= &\max_l \int_0^{t_1} | \rho_{sh}  \langle ( \D^W G_l)(\om),h \rangle_{\h_0} |ds <\infty.
  \tionl

To obtain \eqref{rac-aim} we rewrite \eqref{sum-of-F}
in the following way
\equa
&& \les \rho_{th}  F(\om,G(\om)) -  F(\om,G(\om)) \\
&=&  \int_0^t  \rho_{sh} [\langle  (\D^W F)(\om,G(\om)),h\rangle_{\h_0}  +  \langle \nabla_y F(\om, G(\om)) \cdot   (\D^W G)(\om),h\rangle_{\h_0} ]ds  \\
&& +  \sum_{k=1}^n   \text{ remainder terms}.
\tion
The remainder terms are given by
\equa
&&  \int_{s_{k-1}^t}^{s_k^t}   \rho_{sh} \langle  (\D^W F)(\om,G(\rho_{ (s_k^t -s)h} (\om)))-  (\D^W F)(\om,G(\om)),h\rangle_{\h_0} ds  \\
&& +   \int_{s_{k-1}^t}^{s_k^t} \{ \rho_{s^t_{k-1}h} \nabla_y F( (\om), G(\om)+ \theta  [\rho_{(s^t_k-s^t_{k-1}) h} G(\om) -G(\om)]) \\
&& \quad \quad \quad \quad- \rho_{s^t_{k-1}h} \nabla_y F  (\om, G(\om))  \}                \cdot  \rho_{s h}\partial_hG(\om) ds   \\
&& +  \int_{s_{k-1}^t}^{s_k^t}  \{ \rho_{s^t_{k-1}h} \nabla_y F  (\om, G(\om)) -  \rho_{sh} \nabla_y F  (\om, G(\om)) \} \rho_{sh} \partial_hG (\om)ds \\
&=& I_1 +I_2+I_3,
\tion
where we use  $\partial_hG(\om):= \langle  (\D^W G)(\om),h\rangle_{\h_0} $ as an abbreviation.
It is sufficient to show that the sum of the remainder terms tends in probability to zero for  a fixed sequence of partitions with $\|(s_k^{t,n})\|:= \max_{1 \le k\le n}| s_k^{t,n} -s_{k-1}^{t,n}| \to 0.$
Because of \eqref{Y-h-a-c}, for arbitrary $ \vare_1>0$ one can choose $\delta_1$ and $\|(s_k^{t,n})\|$ sufficiently small such that for  all  $s \in [s_{k-1}^{t,n}, s_{k}^{t,n}]$
\equa
 |\rho_{s_k^{t,n} h} G_l(\om) -\rho_{s h}G_l(\om) | \le \int_s^{s^{t,n}_k} | \langle\rho_{rh} ( \D^W G_l)(\om),h \rangle| dr < \delta_1
\tion
and, by the choice of $\delta_1$ and the continuity of $\delta$ at zero, 
\equa
\delta(|\rho_{s_k^{t,n} h} G_l(\om) -\rho_{s h}G_l(\om) |)\le \vare_1.
\tion
For $\om$  with $ \sup_{s_{k-1}^{t,n}\le s\le s_{k}^{t,n}} |\rho_{sh}G(\om)|\le N $ assumption \eqref{C-three} implies
\equa && \les| \langle  (\D^W F)(\rho_{sh}(\om),G(\rho_{s_k^{t,n} h} (\om)))-  (\D^W F)(\rho_{sh}(\om),G(\rho_{sh}(\om))),h\rangle_{\h_0}|\\
&\le& \|h\|_{\h_0} K_N(\rho_{sh}(\om)) \, \delta(|G(\rho_{s_k^{t,n} h} (\om))-G(\rho_{sh}(\om))|) \\
&\le& \|h\|_{\h_0} K_N(\rho_{sh}(\om)) \, \vare_1.
\tion
Since Lemma \ref{cameron-martin} (ii) implies that $\int_{{ s_{k-1}^{t,n}}}^{s_k^{t,n}}\rho_{sh} K_N(\om)ds < \infty $ a.s. we have $I_1 \to 0$ a.s. for
{$\|(s^{t,n}_k)\| \to 0$}.
\bigskip \\
To estimate $I_2$ we conclude from assumption \eqref{C-one} that
for a.a.~$\om$ it holds for all $n$ and $k=1,...,n$ that $F( \rho_{s^{t,n}_{k-1}h}\om,\cdot) \in  \mathcal{C}^1(\R^d).$ For any such $\om$ and arbitrary $ \vare_2>0$ we have
\equa
 |\rho_{s^{t,n}_{k\!-\!1}h}\{\nabla_y F(\om, G(\om)+ \theta[\rho_{(s^{t,n}_k -s^{t,n}_{k\!-\!1}) h} G(\om) \!-\! G(\om)]) \!-\!  \nabla_y F  (\om, G(\om)) \}| \! < \! \vare_2
\tion
if only $\|(s^{t,n}_k)\|$ is small enough.\bigskip

For the remaining integral $I_3$ we proceed as follows: 
By assumption \eqref{C-four}  there exists  a number $r>q'$ such that $\frac{\partial}{\partial y_k } F(\cdot , G )\in\mathrm{L}_{r}(\PP^W;E)$ for $1 \le k \le d.$
Choose $q_1 \in (1, q)$ and $q_1' \in (q',r)$ such that $\frac{1}{q_1} + \frac{1}{q_1'}=1.$ 
By Lemma \ref{cameron-martin} \eqref{iii} the map
\equa
     [0,t_1] \ni s \mapsto  \rho_{sh}   \nabla_y F(\cdot , G )  \in \mathrm{L}_{q_1'}(\PP^W;E)^d 
     \tion
is uniformly continuous.  From the proof of Lemma \ref{cameron-martin} \eqref{ii} (without integrating over $s$) one can see  that there is some constant $C_h=C(h, q,q_1)>0$
such that  
\[\sup_{0\le s \le t_1} \left\| \rho_{sh} (\partial_h  G)\right\|_{\mathrm{L}_{q_1}(\PP^W;E)}   \le  C_h \left\|  (\partial_h  G)\right\|_{\mathrm{L}_{q}(\PP^W;E)}. 
\]
 Therefore,
\equa
&& \les \E\sum_{k=1}^n   \bigg \| \int_{s_{k-1}^{t,n}}^{ s_k^{t,n}}  \big [  \rho_{s^{t,n}_{k-1}h} \nabla_y F(\cdot, G )    -  \rho_{sh}  \nabla_y F(\cdot , G ) \big ] \cdot
   \rho_{sh}  (\partial_h  G) ds  \bigg \|_E  \\
&\le & \!\sum_{k=1}^n  \!\!  \int_{s_{k-1}^{t,n}}^{ s_k^{t,n}} \! \left\| \| \rho_{s^{t,n}_{k-1}h} \nabla_y F(\cdot , G )    -  \rho_{sh}  \nabla_y F(\cdot , G )\|_E\right\|_{\mathrm{L}_{q'_1}(\PP^W)}\\
   &&    \quad\quad\quad\quad\quad\quad\quad  \times\left\| \|\rho_{sh} (\partial_h  G)\|_E\right\|_{\mathrm{L}_{q_1}(\PP^W)} ds   \\
& \le & \!C_h\sum_{k=1}^n\!\!\int_{s_{k-1}^{t,n}}^{ s_k^{t,n}}\!\left\|  \|\rho_{s^{t,n}_{k-1}h} \nabla_y F(\cdot , G )    -  \rho_{sh}  \nabla_y F(\cdot , G )\|_E\right\|_{\mathrm{L}_{q'_1}(\PP^W)} \\
&&       \quad\quad\quad\quad\quad\quad\quad \times\left\|  \|(\partial_h  G)\|_E\right\|_{\mathrm{L}_{q}(\PP^W)} ds\\
&&  \to  0 \quad  \text{ for }   \|(s^{t,n}_k)\|  \to 0.
\tion
As thus the remainder term $I_3$ tends to $0$ in $\mathrm{L}_{1}(\PP^W;E)$, it also tends to zero in $\mathrm{L}_{0}(\PP)$.   

\end{proof}
\bigskip
\section{Malliavin derivative of solutions to BSDEs}

IIn this section we apply our  theorems on Malliavin differentiability of random functions to generators of BSDEs.  As a result we state in Theorem \ref{diffthm}   that  under conditions on the smoothness of the data $(\xi, f)$  solutions to BSDEs are  Malliavin differentiable.  For simplicity, we set $\sigma=1$ in \eqref{LevyIto}. The  assertions hold  true (with the appropriate modifications) if at least one of them, $\sigma$ or the L\'evy measure $\nu,$ are non-zero.
\bigskip

For $\ 0\leq t\leq T$ we consider the BSDE
\equal \label{beq2}
Y_t&=&\xi+\int_t^T  f\left(X,s,Y_{s}, Z_s,      \int_{\R_0} g( U_s(x))g_1(x) \nu(dx) \right)ds
 -     \int_t^T Z_s   dW_s \non \\
&&\hspace*{12em}-\int_{{]t,T]}\times{\R_0}}U_s(x) \tilde N(ds,dx),
\tionl
with $f\colon D{[0,T]}\times{[0,T]}\times\R^3 \to \R.$ The  conditions on $g$ and $g_1$ are specified  in ($\A_f$\ref{g-condition})  below and ensure that the integral is well-defined.
We use the abbreviations
\equal \label{g-nu}
[g(u)]_\nu:=  \int_{\R_0} g( u(x))g_1(x) \nu(dx)
\tionl
where   $u: \R_0 \to \R$  denotes a measurable function, and
\equa
    f_g(h,s,y, z, u) := f(h,s,y, z, [g(u)]_\nu),
\tion
so that
\[
\int_t^T  f\left(X,s,Y_{s}, Z_s,      \int_{\R_0} g( U_s(x))g_1(x) \nu(dx) \right)ds=\int_t^T  f_g\left(X,s,Y_{s}, Z_s, U_s \right)ds.
\]
The motivation to consider an expression of this form arises from \cite{Morlais} and \cite{Becherer} where BSDEs related to utility maximization have been investigated.
However, to show Malliavin differentiability, our expression had to be chosen in a simpler way.   For the above expression, when  $g$  is the identical map, Malliavin differentiability of
$(Y,Z,U)$ has been stated  in \cite[Theorem 3.5.1]{Delong}. \bigskip \\
For shortness of notation, we define $$\underline{Z}_{s,x}:=\begin{cases}Z_s, &x=0,\\ U_s(x), &x\neq 0\end{cases}$$ to write
$$\int_t^T Z_s   dW_s +\int_{{]t,T]}\times{\R_0}}U_s(x) \tilde N(ds,dx)=\int_{{]t,T]}\times\R} \underline{Z}_{s,x} M(ds,dx).$$
For the terminal value $\xi$ and the function $f_g$ we agree upon the following assumptions:  \bigskip\\
($\A_{\xi}$)  $\xi \in \mathbb{D}_{1,2}$.  \bigskip \\
($\A_f$) \hspace*{2em} \vspace*{-2.0em} \begin{enumerate}[\quad \quad a)]
\item  \label{f-meas}
 $f\colon D{[0,T]}\times{[0,T]}\times\R^3 \to \R$ is jointly measurable, adapted to $(\mathcal{G}_t)_{t\in{[0,T]}}$ defined in  \eqref{filtrationG-t}.
  \item   \label{f(000)-Ltwo}  $ \E\int_0^T\left|f(X,t,0,0,0)\right|^2dt<\infty.$
\item   \label{f-Lip}
$f(X,.,.) \in \mathcal{C}([0,T]\times\R^3) \,\, \PP$-a.s. and
 $f$ satisfies the following Lipschitz condition: There exists a constant $L_f$ such that for all  $t\in{[0,T]}, \eta, \tilde\eta\in\R^3$
 $$ |f\left(X,t, \eta \right)-f\left(X,t,\tilde \eta\right)|\leq L_f  |\eta- \tilde \eta|,$$  $\mathbb{P}$-a.s.
 \item \label{f-yzu-diff}
For all $t\in [0,T]$  and  $i=1,2,3,$ \,$\exists  \,\, {\partial_{\eta_i}}f(X,t,\eta)$  $\mathbb{P}$-a.s. and the functions
\equa
[0,T] \times \R^3 \ni (t,\eta)   \mapsto {\partial_{\eta_i}}f(X,t,\eta)
\tion
are  $\mathbb{P}$-a.s. continuous.

 \item \label{f-Mall-diff-Lip}
$f(X,t,\eta)\in\mathbb{D}_{1,2}$  for all $(t,\eta)\in{[0,T]}\times\R^3,$ and  $\forall t\in{[0,T]}$, there is a function $\delta^t\colon{[0,\infty[}\to{[0,\infty[}$, continuous at zero, with the property that $\forall N\in\N $ $\
\exists\ K^t_N\in\bigcup_{p> 1}\mathrm{L}_p$ such that  for a.a. $\omega$
\begin{equation*}
\begin{split}
&  \forall \eta, \tilde \eta\in B_N(0): \\
&\|\left(\D_{.,0}f(X,t,  \eta)\right)(\omega)-\left(\D_{.,0}f(X,t,  \tilde \eta)\right)(\omega)\|_{\h_0}
   <K^t_N(\omega)\delta^t(\left|\eta-  \tilde \eta\right|),
\end{split}
\end{equation*}
where for $\D_{.,0}f(X,t,  \eta)$ we always take a progressively measurable version in $t$.
\item   \label{f-Mall-diff-bounded}  Assume  there is a random field $\Gamma\in\mathrm{L}_2(\PP \otimes \m)$, such that
 for all random vectors $G \in (\mathrm{L}_2)^3$   and      for a.e. $t$ it holds
\equa
 \left|\left(\D_{s,x}f\right)(t,G)\right|\le \Gamma_{s,x}, \quad  \mathbb{P}\otimes\m \text {-}a.e.
\tion
where $\left(\D_{s,x}f\right)(t,G) := \D_{s,x}f(X,t,\eta)\mid_{\eta=G}.$
\item  \label{g-condition}  $g \in \mathcal{C}^1(\R)$ with bounded derivative and  $g_1\in\mathrm{L}_2(\R_0 , \mathcal{B}(\R_0),\nu).$
\end{enumerate}
A triple  $(Y,Z,U) \in \mathcal{S}_2\times \Ltwo(W) \times \Ltwo(\tilde{N})$ which satisfies (\ref{beq2}) is called a solution to the BSDE (\ref{beq2}).
\begin{remark} \begin{enumerate}
\item
For a function $F\colon\Omega\times{[0,T]}\times\R^3 \to \R$ being jointly measurable, adapted to $(\mathcal{F}_t)_{t\in{[0,T]}}$  one can always find a function $f$ as in ($\A_f $\ref{f-meas}),
such that $\mathbb{P}$-a.s. the equation $$F(\omega,\cdot,\cdot)=f(X(\omega),\cdot,\cdot)$$ holds. Furthermore,  for all $t\in{[0,T]}$, $\eta\in\R^3$     the equation
\[F(\omega,t,\eta)=f\left(X^t(\omega),t,\eta\right)\ \mathbb{P}\text{-a.s.}\] is satisfied  ( for a proof  see \cite[Theorem 4.9.]{delzeith},
\cite[Lemma 3.2., Theorem 3.3.]{Steinicke},  and for the notation  $X^t(\omega)$  recall \eqref{cut-off-in-t}). In particular, for functions satisfying ($\A_f$\ref{f-meas}) it holds $$f\left(h^t,t,\eta\right)=f\left(h,t,\eta\right)\ \mathbb{P}_X\text{-a.s.}$$
for all $t\in{[0,T]}$, $\eta \in\R^3.$
\item Assumption  (${\A}_f$\ref{f-Mall-diff-bounded}) is, in fact, stronger than needed in Theorem \ref{diffthm} below. It is enough to require
that  $  \left|\left(\D_{s,x}f\right)(t,G)\right|\le \Gamma_{s,x}, \,\,  \mathbb{P}\otimes\m \text {-}a.e.$
holds for the solution $G= (Y_t , Z_t, U_t)$ and for the members $G= (Y^n_t , Z^n_t, U^n_t)$ of the approximating sequence appearing  in the proof
of Theorem \ref{diffthm}.  With this more general assumption one can study, for example,  BSDEs with linear generators with random coefficients.

\item The assumption (${\A}_f$\ref{g-condition}) on $g$ can be extended to a dependency on $t$ and $\om.$ Also $g_1$ may be assumed to be time-dependent. To keep the same proof of
Theorem \ref{diffthm} feasible, we have to impose conditions (${\A}_f$\ref{f-meas}-\ref{f-Mall-diff-bounded}) on $g$ (with  $\R^3$ replaced by $\R$ as $g$ is then a random process with
one parameter). Furthermore, we have to assume that $g_1\colon [0,T]\times\R_0\to\R$ is Borel measurable and that $\|g_1(t,.)\|_{\Ltwo(\nu)}$ is bounded in $t\in [0,T]$.
\end{enumerate}
\end{remark}

To cover the issue of existence of  solutions to BSDEs we refer to the following result:
\begin{theorem}[\cite{Tangli}, Lemma 2.4] \label{existence}
Assume $(\xi,f)$ satisfies the assumptions $\xi\in\Ltwo$ and  (${\A}_f$\ref{f-meas}-\ref{f-Lip}). Then the BSDE (\ref{beq2}) has a  unique solution $(Y,Z,U) \in
 \mathcal{S}_2\times \Ltwo(W) \times \Ltwo(\tilde{N}).$
\end{theorem}

We cite the stability result of Barles,  Buckdahn and Pardoux (\cite{bbp}) comparing the distance between solutions to the BSDE (\ref{beq2}) with different
terminal conditions and generators.

\begin{theorem}[\cite{bbp}, Proposition 2.2] \label{continuitythm}
Assume that  $(\xi,f_g)$ and $(\xi',f'_g)$  satisfy $\xi,\xi'\in\Ltwo$ and suppose the generators fulfill  (${\A}_f$\ref{f-meas}-\ref{f-Lip}), while  $g$  is Lipschitz and  $g_1\in\mathrm{L}_2(\R_0 , \mathcal{B}(\R_0),\nu).$ Then there exists a constant $C>0$ such that for the corresponding solutions
 $(Y,Z,U)$ and $(Y',Z',U')$
to (\ref{beq2}) it holds
\equa
&&  \|Y-Y'\|_{S}^2+ \|Z-Z'\|_{ \Ltwo(W)}^2  +  \|U-U'\|_{ \Ltwo(\tilde{N})}^2  \\
&\le & C\bigg ( \|\xi-\xi'\|_{\Ltwo}^2
  +\int_0^T\|f_g\left(X,s,Y_{s}, Z_s, U_s \right)-f'_g\left(X,s,Y_{s}, Z_s, U_s  \right)\|_{\Ltwo}^2 ds \bigg).
\tion
\end{theorem}


\bigskip
We state now the result about the Malliavin derivative of solutions to BSDEs.  For the proof  we apply It\^o's formula like in the original work due to  Pardoux and Peng  \cite{PP1}  or in Ankirchner et al. \cite{AnImRe}.
The benefit is that one does not need any higher moment conditions on the data than $\Ltwo.$ Hence this result is a generalization of  El Karoui et al. \cite[Theorem 5.3]{ElKarouiPengQuenez}. 
It is also more general than  \cite[Theorem 3.5.1]{Delong} of Delong: For example, we do not require a canonical L\'evy space, the L\'evy process does not need to be square integrable, and   the generator in \eqref{beq2} allows some nonlinear structure w.r.t. $U_s(x)$   thanks  to the function $g$.

\begin{theorem}\label{diffthm}
Assume (${\A}_\xi$) and (${\A}_f$).  Then the following assertions hold.
\begin{enumerate}[(i)]
\item \label{alef}  For $\m$- a.e. $(r,v) \in [0,T]\times \R$  there exists a unique solution $(\mathcal{Y}^{r,v},\mathcal{Z}^{r,v}, \mathcal{U}^{r,v}) $ $\in  \mathcal{S}_2\times \Ltwo(W) \times \Ltwo(\tilde{N})$ to the BSDE
\equal  \label{U-V-equation}
\mathcal{Y}^{r,v}_t &=& \D_{r,v} \xi + \int_t^T F_{r,v}\kla s, \mathcal{Y}^{r,v}_s, \mathcal{Z}^{r,v}_s,  \mathcal{U}^{r,v} _s\mer ds  \non\\ &&
  -\int_{{]t,T]}\times\R} \underline{\mathcal{Z}}^{r,v}_{s,x}M(ds,dx), \quad 0\le r \le t \le T \non \\
  \mathcal{Y}^{r,v}_s &=& \mathcal{Z}^{r,v}_s =   \mathcal{U}^{r,v}_s = 0,    \quad   0\le s< r\le T ,
\tionl
where
\equa
\underline{\mathcal{Z}}^{r,v}_{s,x} := \left \{ \begin{array}{ll} \mathcal{Z}^{r,v}_s, &x=0 \\
 \mathcal{U}^{r,v} _s(x), &x\neq 0,
\end{array} \right .
\tion
and
\equa
 &&\les  F_{r,v} (s, y,z,u)\!:=\! \left\{ \begin{array}{ll}
\!\!\!\!(\D_{r,0} f_g)\left(X, s, Y_s,Z_s,U_s\right) & \smallskip\\
\!\!\!\!+\!\left\langle\nabla f\!\left(X,s,Y_s, Z_s, [g(U_s)]_\nu \right)\!, (y, z, [g'(U_s) u  ]_\nu \, )\right\rangle\!, &\!\!\! v=0 \\
 & \\
\!\!\!\!f_g\!\left(X+v \one_{[r,T]},s,Y_s+y, Z_s+z, U_s + u
 \right) & \smallskip\\
\!\!\quad\quad -f_g\left(X,s,Y_s, Z_s, U_s\right)\!,  &\!\!\! v\neq 0, \\
\end{array} \right.
\tion
with $\nabla =(\partial_{\eta_1}, \partial_{\eta_2}, \partial_{\eta_3}).$

\item \label{bet} For the solution $(Y,Z,U)$ of \eqref{beq2} it holds
\equal \label{Y-and-Z-in-D12}
Y, Z   \in \Ltwo([0,T];\DD), \quad U\in \Ltwo([0,T]\times\R_0;\DD),
\tionl
and $\D_{r,y}Y$ admits a c\`adl\`ag version for $\m$- a.e. $(r,y) \in  [0,T]\times\R.$
\item \label{gimmel} $(\D Y,\D Z, \D U)$ is a version of $(\mathcal{Y},\mathcal{Z}, \mathcal{U}),$  i.e. for $\m$- a.e. $(r,v)$ it solves
\begin{align} \label{diffrep}
\D_{r,v}Y_t = &\D_{r,v}\xi+\int_t^T F_{r,v}\left(s,\D_{r,v}Y_s, \D_{r,v}Z_s,\D_{r,v} U_s \right)ds \\
&-\int_t^T \D_{r,v}Z_sdW_s-\int_{{]t,T]}\times{\R_0}} \D_{r,v} U_{s}(x)\tilde{N}(ds,dx), \quad 0\le r \le t \le T. \non
\end{align}
\item \label{dalet} Setting  $D_{r,v}Y_r(\omega):= \lim_{t\searrow r}\D_{r,v}Y_t(\omega) $
for all $(r,v,\omega)$ for which  $ \D_{r,v}Y$ is c\`adl\`ag and $\D_{r,v}Y_r(\omega):=0$ otherwise, we have

\begin{align*} {\phantom{\int}}^p\left(\left(D_{r,0}Y_r\right)_{r\in{[0,T]}}\right) &\text{ is a version of  }(Z_{r})_{r\in{[0,T]}},\\
{\phantom{\int}}^p\left(\left(D_{r,v}Y_r\right)_{r\in{[0,T]}, v\in\R_0}\right) &\text{ is a version of  }(U_{r}(v))_{r\in{[0,T]}, v\in\R_0}.
\end{align*}

\end{enumerate}

\end{theorem}

\bigskip

We present an example  of  a FBSDE where we specify the dependence on $\om$ in the  generator by a forward process  such that (${\A}_f$\ref{f-Mall-diff-bounded}) holds.
\begin{example}
 Consider the case of a L\'evy process $X$ such that $\E|X_t|^2<\infty$ for all $t\in{[0,T]}$.   Assume the generator to be of the type $$f(s,\omega,y,z,u)=\tilde f(s, \Psi_s(\omega),y,z,u),$$ 
with $\tilde f$ having a  continuous partial derivative in the second variable bounded by $K$. Assume further that this partial derivative is locally Lipschitz in $(y,z,u)$. Let $\Psi$ denote a forward process given by the SDE
$$d\Psi_s=b(\Psi_s)ds+\sigma(\Psi_s)dW_s+\beta(\Psi_{s-},x)\tilde{N}(ds,dx)$$
with $\Psi_0 \in \R.$
Then conditions (${\A}_f$\ref{f-Mall-diff-Lip}), (${\A}_f$\ref{f-Mall-diff-bounded}) are satisfied under the requirements
\begin{enumerate}
\item[(i)] The functions $b\colon\R\to\R$ and $\sigma\colon\R\to\R$ are continuously differentiable with bounded derivative.
\item[(ii)] $\beta\colon\R\times\R_0\to\R$ is measurable, satisfies
 \begin{align*}
&\left|\beta(\psi,x)\right|\leq C_\beta(1\wedge|x|),\quad (\psi,x)\in\R\times\R_0, \\
& \left|\beta(\psi,x)  -  \beta(\hat \psi,x)  \right|\leq C_\beta | \psi-\hat\psi|  (1\wedge|x|),\quad (\psi,x), (\hat\psi,x)   \in\R\times\R_0, 
\end{align*}
and is continuously differentiable in $\psi$ 
for fixed $x\in\R_0.$
\end{enumerate}
This follows, since $\left(\D_{s,x}f\right)(t,G)$ is given by
$$\left(\D_{s,x}f\right)(t,G)=\begin{cases}\tilde{f}_\psi(t,\Psi_t,G)\D_{s,x}\Psi_t, & x=0,\\
\tilde{f}(t,\Psi_t+\D_{s,x}\Psi_t,G)-\tilde{f}(t,\Psi_t,G), & x\neq 0, \end{cases}$$
implying
$$\left|\left(\D_{s,x}f\right)(t,G)\right|< K\left|\D_{s,x}\Psi_t\right|.$$
Theorem \cite[Theorem 4.1.2]{Delong} states that under the above conditions on $b,\sigma$ and $\beta$,  $$\sup_{r,v}\E\sup_{s\in{[0,T]}}\left|\frac{\D_{r,v}\Psi_s}{v}\right|^2<\infty,$$
and refers to \cite[Theorem 3]{Petrou} for a proof. 
Thus, to satisfy (${\A}_f$\ref{f-Mall-diff-bounded}),  we may choose $\Gamma=C\sup_{s\in{[0,T]}}\left|\D\Psi_s\right|$, where $C$ depends on $K,C_\beta$ and the Lipschitz constants for $b,\sigma$ and $\beta$.

\end{example}

\subsection{Proof of  Theorem \ref{diffthm}}\label{4}

  Let us start with a lemma providing estimates for the Malliavin derivative of the generator.

\begin{lemma}\label{destimlem}
Let $G=(G_1,G_2,G_3) \in (\Ltwo)^3$ and $\Phi\in(\Ltwo(\PP\otimes\m))^3$. If $f$ satisfies (${\A}_f$)
it holds for $\mathbb{P} \otimes \m$-a.a. $(\omega,r,v), v \neq 0,$ that
\equal \label{generator-estimate}
&& \les \bigg |f(X+v\one_{[r,T]},t,G+\Phi_{r,v})
-f\left(X,t,G\right)\bigg | \le   L_f \left|\Phi_{r,v}\right|+ \Gamma_{r,v}.
\tionl
Moreover, for  $G \in (\DD)^3$
it holds   $f(X,t,G) \in \DD$ and
\equal \label{Dgenerator-estimate}
  |\D_{r,v}f\left(X,t,G\right)|   \le   L_f \left|D_{r,v} G\right| +\Gamma_{r,v}, \quad \mathbb{P}\otimes \m \text{-a.e.}
\tionl
\end{lemma}
\begin{proof}
According to Corollary  \ref{kept-almost-sure}  we may replace $X$ by  $X+v\one_{[r,T]}$ and use the  Lip\-schitz property (${\A}_f \ref{f-Lip}$) to estimate
$$ \big | f(X+v\one_{[r,T]},t,G+\Phi_{r,v}) - f(X+v\one_{[r,T]},t,G)  \big | \le    L_f \left|\Phi_{r,v}\right|$$
for  $\mathbb{P}\otimes \m $-a.e. $(\om, r,v)$ with $v\neq 0.$
From (${\A}_f \ref{f-Mall-diff-bounded}$) one concludes then \eqref{generator-estimate}.

For $v \neq 0$  we  conclude from Lemma \ref{functionallem1} that  $\D_{r,v}f\left(X,t,G\right)   \in   \mathbb{D}_{1,2}^{\R_0}$ and  apply Lemma \ref{functionallem} to get $$\D_{r,v}f\left(X,t,G\right)=  f(X+v\one_{[r,T]},t,G+\D_{r,v}G)     -  f\left(X,t,G\right),$$
and hence \eqref{Dgenerator-estimate} follows from \eqref{generator-estimate}.
 In the case of $v=0$, by assumption (${\A}_f \ref{f-Mall-diff-Lip}$)    we may apply Theorem \ref{Mall-diff-thm}. Thus we get the Malliavin derivative
\equal \label{D0generator}
\D_{r,0}f(X,t,G) &=&   (\D_{r,0}f)(t,G) + \partial_{\eta_1}f (X,t,G)\D_{r,0} G_1 \non \\
&& +   \partial_{\eta_2}f(X,t,G)\D_{r,0} G_2  +\partial_{\eta_3} f(X,t,G)\D_{r,0} G_3
\tionl
for $\mathbb{P}\otimes \lambda$ a.a. $(\omega,r)\in \Omega\times{[0,T]}.$ Relation \eqref{Dgenerator-estimate} follows from conditions (${\A}_f \ref{f-Lip}$)
 and  \eqref{f-Mall-diff-bounded} using
that the partial derivatives are bounded by $L_f$.
\end{proof}

{\bf Proof of Theorem \ref{diffthm}.}
The core of the proof is to conclude assertion \eqref{bet} which will be done by an iteration argument.
To simplify the notation we do not  mention the dependency of $f$ on $X$ in most places.  \bigskip

\eqref{alef} For those $(r,v)$ such that  $\D_{r,v}\xi \in \Ltwo$  the existence and uniqueness of a solution $(\mathcal{Y}^{r,v},\mathcal{Z}^{r,v},\mathcal{U}^{r,v})$ to \eqref{U-V-equation} follows from Theorem \ref{existence} since $F_{r,v}$ meets the assumptions of the theorem.\\
\eqref{bet} By  Theorem \ref{continuitythm} the  solution depends continuously on the terminal condition and  $\D\xi$ is measurable  w.r.t.~$(r,v).$
We infer the measurable dependency  $(r,v) \mapsto(\mathcal{Y}^{r,v},\mathcal{Z}^{r,v},\mathcal{U}^{r,v})$  as follows:  Since by  Theorem \ref{continuitythm}  the mapping
\[
\Ltwo \to  \mathcal{S}_2\times \Ltwo(W) \times \Ltwo(\tilde{N})  \colon \xi \mapsto (Y,Z,U)
\]
is continuous one can show the existence of a jointly measurable version of
 \[
 (\mathcal{Y}^{r,v},\mathcal{Z}^{r,v}, \mathcal{U}^{r,v}), \quad  (r,v) \in [0,T]\times \R \]
  by approximating  $\D\xi$ with simple functions in $\Ltwo(\PP \otimes\m).$ Joint measurability (for example for $\mathcal{Z}$) in all arguments can be gained by identifying
  the spaces
$$\Ltwo(\lambda,\Ltwo(\mathbb{P}\otimes\m))\cong\Ltwo(\lambda \otimes\mathbb{P}\otimes\m).$$  The quadratic integrability with respect to $(r,v)$  also follows  from Theorem
\ref{continuitythm} since $\xi \in \DD .$ \bigskip

Using an iteration scheme, starting with $(Y^0,Z^0,U^0)=(0,0,0)$, we get $Y^{n+1}$ by taking the optional projection which implies that
\equal \label{Yn}
Y^{n+1}_t=\E _{t}\left(\xi+\int_t^T f_g\left(s,Y^n_{s}, Z^n_s, U_s^n\right)ds\right).
\tionl
 The process $\underline{Z}^{n+1}$ given by
\equa
\underline{Z}^{n+1}_{s,x} := \left \{ \begin{array}{ll} Z ^{n+1}_s, &x=0, \\
U^{n+1}_s(x), &x\neq 0,
\end{array} \right .
\tion
 one gets by the martingale representation theorem w.r.t. $M$ (see, for example,  \cite{Applebaum}):
\equal   \label{zndef}
\xi+\int_0^T f_g\left(s,Y^n_s, Z^n_s, U^n_s\right)ds&=&
\E \left(\xi+\int_0^T f_g \left(s,Y^n_{s}, Z^n_s, U^n_s\right)ds\right) \non \\
&&+\int_{{]0,T]}\times\R}\underline{Z}^{n+1}_{s,x}M(ds,dx).
\tionl

{\tt Step 1.} \,
It is well-known that  $(Y^n, Z^n, U^n)$  converges to the solution  $(Y,Z,U)$ in $\Ltwo(W)\times \Ltwo(W) \times \Ltwo(\tilde N).$ Our aim in this step
is to show that $Y^n, Z^n$ and $U^n$  are  uniformly bounded in $n$ as elements of   $\Ltwo(\lambda ; \DD)$ and $\Ltwo(\lambda\otimes \nu  ; \DD),$ respectively.
This will follow  from \eqref{boundedYn} below.   \\
 Given that $Y^n , Z^n \in \Ltwo( \lambda; \DD)$ and $U^n \in \Ltwo(\lambda\otimes \nu; \DD)$ one can infer that this also
 holds for $n+1\colon$   Indeed, (${\A}_f \ref{g-condition}$) implies that $[g(U^n_s)]_\nu \in \DD$ for a.e. $s$ and
 \equal \label{g-integral-est}
   \left | \D_{r,v} [g(U^n_s)]_\nu \right |  \le L_g \|g_1\|_{\Ltwo(\nu)}  \| \D_{r,v} U^n_s   \|_{\Ltwo(\nu)}.
  \tionl
 From Lemma  \ref{destimlem} we get that $f(X,s,Y^n_s, Z^n_s,  [g(U^n_s)]_\nu) \in \DD.$  The above estimate and \eqref{Dgenerator-estimate}  as well as the Malliavin differentiation rules shown by Delong and Imkeller in  \cite[ Lemma 3.1. and Lemma 3.2.]{ImkellerDelong}  imply that  $Y^{n+1}$ as defined in \eqref{Yn}   is   in   $\Ltwo( \lambda; \DD).$  Then we conclude from \eqref{zndef} and
 \cite[ Lemma 3.3.]{ImkellerDelong}  that $Z^{n+1} \in \Ltwo( \lambda; \DD)$ and $U^{n+1} \in \Ltwo(\lambda\otimes \nu; \DD).$
 Especially,  we get for $t\in{[0,T]}$ that  $\PP$ -a.e.
\equal  \label{D-of-y-n+1}
\D_{r,v}Y^{n+1}_t&=&\D_{r,v}\xi+\int_t^T \D_{r,v}f_g\left(X,s,Y^n_s, Z^n_s, U^n_s \right)ds  \non\\
&&-\int_{{]t,T]}\times\R}\D_{r,v}\underline{Z}^{n+1}_{s,x}M(ds,dx)   \text{ for } \m \text{ - } a.a. \,(r,v) \in [0,t]\times \R , \non \\
\D_{r,v}Y^{n+1}_t&=&0  \quad \text{ for } \m \text{ - } a.a. \,(r,v) \in (t,T]\times \R,  \non \\
\D_{r,v} \underline{Z}^{n+1}_{t,x}&=&0  \quad \text{ for } \m\otimes\mu \text{ - } a.a. \,(r,v,x) \in (t,T]\times \R^2.
\tionl

Since by \cite[Theorem 4.2.12] {Applebaum}  the process $\big (\int_{{]0,t]}\times\R}\D_{r,v}\underline{Z}^{n+1}_{s,x}M(ds,dx)\big)_{t\in [0,T]}$ admits a c\`adl\`ag version, we may take a c\`adl\`ag version of both sides. \\
By It{\^o}'s formula (see, for instance,  \cite{Applebaum}), we conclude that for $0< r<t$ it holds
\equa
e^{\beta T}(\D_{r,v}\xi)^2&=&e^{\beta t}(\D_{r,v}Y^{n+1}_t)^2\\
&&+\beta\int_t^T e^{\beta s}(\D_{r,v}Y^{n+1}_s)^2 ds \\
&&-2\int_t^T e^{\beta s}\big [\D_{r,v}f_g\left(X,s,Y^n_s,Z^n_s, U^n_s \right)\big]\D_{r,v}Y^{n+1}_sds\\
&&+\int_{{]t,T]}\times\R} e^{\beta s}[2(\D_{r,v}Y^{n+1}_{s-} )\D_{r,v}\underline{Z}^{n+1}_{s,x} \\
&& \quad \quad \quad   \quad \quad \quad   \quad          +  \one_{\R_0}(x)  ({\D_{r,v}\underline{Z}^{n+1}_{s,x}})^2 ]M(ds,dx)\\
&&+\int_{{]t,T]}\times\R}e^{\beta s}(\D_{r,v}\underline{Z}^{n+1}_{s,x})^2ds\mu(dx) \quad \mathbb{P}\otimes\m\text{ - } a.e.
\tion

One easily checks that the integral w.r.t.~$M$ is a uniformly integrable martingale and hence has expectation zero. Therefore, using \eqref{D-of-y-n+1}, we have for $0< u<t \le T$ that
\equal  \label{expecteditody2}
&& \les\E e^{\beta t}(\D_{r,v}Y^{n+1}_t)^2+\E\int_{{]r,T]}\times\R}e^{\beta s}(\D_{r,v}\underline{Z}^{n+1}_{s,x})^2ds\mu(dx)  \non\\
&\le&
e^{\beta T }\E(\D_{r,v}\xi)^2+ 2\int_r^T e^{\beta s}\E\left| \big [ \D_{r,v}f_g\left(X,s,Y^n_s,Z^n_s, U^n_s\right)\big]\D_{r,v}Y^{n+1}_s\right|ds  \non\\
&&-\beta\E\int_r^T e^{\beta s}(\D_{r,v}Y^{n+1}_s)^2 ds.
\tionl

 By Young's inequality, \eqref{g-integral-est} and
 Lemma  \ref{destimlem}  we get a constant $C_f$ such that for any $c>0$,
\equa
&&  \les 2 \left| \big [ \D_{r,v}f_g\left(X,s,Y^n_s,Z^n_s, U^n_s\right) \big ]\D_{r,v}Y^{n+1}_s\right|  \\
&  \le & c\left| \D_{r,v}Y^{n+1}_s\right|^2        +\frac{C_f}{c}  \Big(   \left|\Gamma_{r,v}\right|^2
+\left|\D_{r,v}Y^{n}_s\right|^2+\left|\D_{r,v}Z^n_s\right|^2  +  \| \D_{r,v} U^n_s   \|_{\Ltwo(\nu)}^2   \Big) \\
&  = & c\left| \D_{r,v}Y^{n+1}_s\right|^2        +\frac{C_f}{c}  \Big(   \left|\Gamma_{r,v}\right|^2
+\left|\D_{r,v}Y^{n}_s\right|^2 +   \int_{\R} \left| \D_{r,v} \underline{Z}^n_{s,x}  \right|^2  \mu(dx)   \Big). \\
\tion

Choosing $\beta=c+1$  and $c= 2 C_f$ leads to
\equa
&& \les \E\int_r^T e^{\beta s}\left|\D_{r,v}Y^{n+1}_s\right|^2 ds  +\E\int_{{]r,T]}\times\R}e^{\beta s}\left|\D_{r,v}\underline{Z}^{n+1}_{s,x}\right|^2\m(ds,dx)\\
&\leq& e^{\beta T}\E\left|\D_{r,v}\xi\right|^2+\frac{1}{2}\int_r^T e^{\beta s}ds \,\E \left|\Gamma_{r,v}\right|^2\\
&&+\frac{1}{2}\left(\E\int_r^T e^{\beta s}\left|\D_{r,v}Y^{n}_s\right|^2ds+\E\int_{{]r,T]}\times\R} e^{\beta s} \left|\D_{r,v}\underline{Z}^{n}_{s,x}\right|^2 \m(ds,dx)\right).
\tion
Finally, \eqref{D-of-y-n+1} and Lemma \ref{gronwalllemma2}  imply
\equal  \label{boundedYn}
&& \les \int_0^Te^{\beta s} \| \D Y^{n}_s \|_{L_2(\m\otimes \mathbb{P})}^2ds + \int_{{[0,T]}\times\R}e^{\beta s}\|\D \underline{Z}^{n}_{s,x}\|_{L_2(\m\otimes \mathbb{P})}^2\m(ds,dx) \non\\
 && \le c_{\beta} \| |\D\xi| +\Gamma  \|_{\Ltwo(\mathbb{P} \otimes \m)}^2 \text{ for all } n \in \N.
\tionl
\smallskip

{\tt Step 2.} \,
We now show that
\equal \label{diffnorm}
\left\|\mathcal{Y}-\D Y^{n+1}\right\|^2_{\mathrm{L}^2(\mathbb{P}\otimes\lambda\otimes\m)}+\left\|\underline{\mathcal{Z}}-\D \underline{Z}^{n+1}\right\|^2_{\mathrm{L}^2(\mathbb{P}\otimes(\m)^{\otimes 2})} \to 0, \quad n \to \infty.
\tionl

In order to estimate the expressions from \eqref {diffnorm} one can repeat the previous computations  for the difference $\mathcal{Y}^{r,v}_t-\D_{r,v} Y^{n+1}_t$ to obtain
\equal \label{expecteditoddiff}
&& \les \E\int_r^T e^{ \beta s}(\mathcal{Y}^{r,v}_s-\D_{r,v}Y^{n+1}_s)^2ds+\E\int_{{]r,T]}\times\R}e^{ \beta s}(\underline{\mathcal{Z}}^{r,v}_{s,x}-\D_{r,v}\underline{Z}^{n+1}_{s,x})^2ds\mu(dx) \non \\
&\le&  \frac{1}{c}\E\int_r^T e^{ \beta s}\left|F_{r,v}(s,\mathcal{Y}^{r,v}_s, \mathcal{Z}^{r,v}_s,  \mathcal{U}^{r,v} _s)-\D_{r,v}f_g(s,  Y^n_s,  Z^n_s, U^n_s ) \right|^2 ds.
\tionl
for any  $c>0.$ \\
For the case $v=0,$ by using Lipschitz properties of $f$ (which also imply the boundedness of the partial derivatives), we can find a constant $C_f'$  such that
\equal
\label{Fr0minusDr0f}
&&  \les \left|F_{r,0}(s,\mathcal{Y}^{r,0}_s, \mathcal{Z}^{r,0}_s, \mathcal{U}^{r,0}_s)-\D_{r,0}f_g(s,  Y^n_s,  Z^n_s, U^n_s )\right| \phantom{ \bigcup } \non \\
&\le& C_f'
 (\big | \mathcal{Y}^{r,0}_s   -\D_{r,0}Y^n_s \big | +  \big | \mathcal{Z}^{r,0}_s   -\D_{r,0}Z^n_s \big | + \big\| \mathcal{U}^{r,0}_s - \D_{r,0}U^n_s \big \|_{\Ltwo(\nu)} ) \non \\
&&+ \kappa_n(r,s)
\tionl
where for some $C>0$

\equal
 \kappa_n(r,s) &= &C \big (\big | (\D_{r,0}f_g)(s,Y_s, Z_s,  U_s)  - (\D_{r,0}f_g) (s,  Y^n_s,  Z^n_s, U^n_s ) \big |\wedge \Gamma_{r,0}\non \\
&&+ \big | \mathcal{Y}^{r,0}_s  \big |   \,  \big |\partial_y  f_g(s,Y_s, Z_s,  U_s)   -\partial_y f_g(s,  Y^n_s,  Z^n_s, U^n_s ) \big |  \non \phantom{\bigcup}  \\
&& +\big| \mathcal{Z}^{r,0}_s \big| \,  \big |\partial_z f_g(s,Y_s, Z_s,  U_s)    -\partial_z f_g(s,  Y^n_s,  Z^n_s, U^n_s )\big|  \non \phantom{\bigcup}\\
&& +\big\| \mathcal{U}^{r,0}_s \big \|_{\Ltwo(\nu)} (  \big |\partial_u  f_g(s,Y_s, Z_s,  U_s)   -\partial_u f_g(s,  Y^n_s,  Z^n_s, U^n_s ) \big|    \non \phantom{\bigcup}\\
&& \quad \quad +  \left \|  |g'(U_s) - g'(U^n_s)| g_1 \right \|_{\Ltwo(\nu)} ) \big).
\tionl
Since
the sequence  $( Y^n, Z^n, U^n  )$ converges in   $\Ltwo(W)\times \Ltwo(W) \times \Ltwo(\tilde N),$  condition  (${\A}_f \ref{f-Mall-diff-Lip}  $)  holds, and    $\partial_yf, \partial_zf,\partial_uf$ as well as $g'$ are bounded and continuous  it follows from  Vitali's convergence theorem that
\equal \label{delta-n-terms}
 \delta_n&:=& \E\int_r^T e^{\beta s}  \kappa_n(r,s)^2  dr
      ds \to 0  \,\,\text{ for } \,\, n\to \infty.
\tionl

Now we continue with the case $v\neq 0.$  We  first  realize  that  for a given $\vare >0$ we may choose  $\alpha >0$ small enough such that

 \equa
 \E\!\!\int_r^T\!\!\! \int_{\{|v| < \al\}} \!\!\!\! e^{ \beta s}\left|F_{r,v}(s,\mathcal{Y}^{r,v}_s\!, \mathcal{Z}^{r,v}_s\!,  \mathcal{U}^{r,v} _s)-\D_{r,v}f(s,  Y^n_s,  Z^n_s, U^n_s ) \right|^2 \nu(dv)ds < \vare.
 \tion
This is because  from  \eqref{generator-estimate},  \eqref{Dgenerator-estimate}  and \eqref{g-nu} one gets by a straightforward calculation
\equa
             \left|F_{r,v}(s,\mathcal{Y}^{r,v}_s, \mathcal{Z}^{r,v}_s ,\mathcal{U}^{r,v}_s)\right|_{\Ltwo(\nu)}
 &\le&  \Gamma_{r,v} + L_f( |\mathcal{Y}^{r,v}_s|+|\mathcal{Z}^{r,v}_s| + L_g [|  \mathcal{U}^{r,v}_s|  ]_\nu     )\\
 &\le& \Gamma_{r,v} + L_{f,g}( |\mathcal{Y}^{r,v}_s|+|\mathcal{Z}^{r,v}_s| + \|\mathcal{U}^{r,v}_s \|_{\Ltwo(\nu)}    )\\
\tion
with $L_{f,g}=L_f (1+L_g  \| g_1 \|_{\Ltwo(\nu)})$ where $L_g$ is the Lipschitz constant of $g,$ and
\begin{equation*}
             \left|\D_{r,v}f_g(s,Y^n_s,Z^n_s,U^n_s)\right|\le  \Gamma_{r,v} + L_{f,g}( | \D_{r,v} Y^n_s|+ | \D_{r,v} Z^n_s|+ \| \D_{r,v} U^n_s\|_{\Ltwo(\nu)}).
\end{equation*}

On the set $\{ |v| \ge \alpha\}$ we  use the Lipschitz properties (${\A}_f \ref{f-Lip} $) and  (${\A}_f \ref{g-condition} $)    to get the estimate

\equa
&&\les  \left|F_{r,v}(s,\mathcal{Y}^{r,v}_s, \mathcal{Z}^{r,v}_s ,\mathcal{U}^{r,v}_s)-\D_{r,v}f_g(s,Y^n_s,Z^n_s,U^n_s)\right| \\
&\le& \big |f_g\left((X+v \one_{[r,T]}),s,Y_s+\mathcal{Y}^{r,v}_s, Z_s+\mathcal{Z}^{r,v}_s, U_s + \mathcal{U}^{r,v}_s
 \right) \\
 && \quad - f_g\left((X+v\one_{[r,T]}),s,Y^n_s + \D_{r,v} Y^n_s, Z^n_s+\D_{r,v} Z^n_s , U^n_s + \D_{r,v} U^n_s \right)  \big | \\
&& + \big | f_g\left(X,s,Y_s, Z_s, U_s\right)
-f_g\left(X,s,Y^n_s, Z^n_s, U^n_s\right) \big |\\
&\le&  L_{f ,g}\big [  | \mathcal{Y}^{r,v}_s -  \D_{r,v} Y^n_s | +|\mathcal{Z}^{r,v}_s-\D_{r,v}Z^n_s|+ \| \mathcal{U}^{r,v}_s - \D_{r,v} U^n_s \|_{\Ltwo(\nu)} \\
&& \quad \quad \quad + 2 (|Y_s- Y^n_s|+|Z_s -Z^n_s|+ \| U_s- U^n_s \|_{\Ltwo(\nu)} ) \big ]  .
\tion
This gives for any  $ n\in \N$

\equa
&& \les \E\!\int_r^T \!\!\int_{[0,T]\times\R} \! e^{ \beta s} |F_{r,v}(s,\mathcal{Y}^{r,v}_s, \mathcal{Z}^{r,v}_s ,\mathcal{U}^{r,v}_s)-\D_{r,v}f_g(s,Y^n_s,Z^n_s,U^n_s)
 |^2 \m(dr,dv)  ds\\
 &\le&  c(L_{f,g}) \E \int_r^T e^{ \beta s} \big (\| \mathcal{Y}_s-  \D Y_s^n \|_{L_2(\m)}^2  +  \|\underline{\mathcal{Z}}_{s,.}-\D \underline{Z}^n_{s,.}\|_{L_2(\m\otimes \mu)}^2  \big )ds \\
 && +       2 c(L_{f,g}) \nu(\{ |v| \ge \al \}) \E \int_r^T e^{ \beta s} \big(|Y_s- Y^n_s|^2+ \|\underline{Z}_{s,.}-\underline{Z}^n_{s,.}\|^2_{L_2(\m\otimes \mu)} \big)ds      \\
&&+ \delta_n+      \varepsilon.
\tion

Choosing  $c$ in  \eqref{expecteditoddiff}  in  an appropriate way leads to
\equa
&&\les \left\|\mathcal{Y}-\D Y^{n+1}\right\|^2_{\mathrm{L}^2(\mathbb{P}\otimes\lambda\otimes\m)}+\left\|\underline{\mathcal{Z}}-\D \underline{Z}^{n+1}\right\|^2_{\mathrm{L}^2(\mathbb{P}\otimes(\m)^{\otimes 2})}\\
&\le & \vare + C_n +\frac{1}{2}\left( \left\|\mathcal{Y}-\D Y^{n}\right\|^2_{\mathrm{L}^2(\mathbb{P}\otimes\lambda\otimes\m)}
  +\left\|\underline{\mathcal{Z}}-\D \underline{Z}^{n}\right\|^2_{\mathrm{L}^2(\mathbb{P}\otimes(\m)^{\otimes 2})}\right)
\tion
with  $C_n=C_n(\alpha)$ tending to zero if $n\to \infty$ for any fixed $\alpha>0.$ We now apply Lemma \ref{gronwalllemma2} and end up with
\begin{equation*}
\begin{split}
&\limsup_{n\to\infty}\left( \left\|\mathcal{Y}-\D Y^{n}\right\|^2_{\mathrm{L}^2(\mathbb{P}\otimes\lambda\otimes\m)}+\left\|\underline{\mathcal{Z}}-\D \underline{Z}^{n}\right\|^2_{\mathrm{L}^2(\mathbb{P}\otimes(\m)^{\otimes 2})}\right) \le 2 \vare.\\
\end{split}
\end{equation*}
This implies \eqref{Y-and-Z-in-D12}. Hence we can take the Malliavin derivative of  \eqref{beq2} and get  \eqref{diffrep} as well as
\equal  \label{the-other-D-equation}
0&= &\D_{r,v}\xi+\int_r^T F_{r,v}\left(s,\D_{r,v}Y_s, \D_{r,v}Z_s,  \D_{r,v} U_s \right)ds \non\\
&& - \underline{Z}_{r,v}-\int_{{]r,T]}\times\R} \D_{r,v} \underline{Z}_{s,x}M(ds,dx), \quad 0\le t  < r \le T.
\tionl
By the same reasoning as for  $\D_{r,v}Y^n$ we may conclude that  the RHS of \eqref{diffrep} has a c\`adl\`ag version which we take for
$\D_{r,v}Y.$\bigskip

\eqref{gimmel} This assertion we get  comparing \eqref{U-V-equation} and  \eqref{diffrep}  because of the  uniqueness of $(\mathcal{Y},\mathcal{Z},\mathcal{U}).$ \bigskip

\eqref{dalet} We first discuss the measurability of $\lim_{t \searrow r}\D_{r,v}Y_t$ w.r.t.~$(r,v,\om)$ which is needed to take the predictable projection.
From \eqref{diffrep} one concludes that for any fixed $(r,v)$ there exists a c\`adl\`ag version of  $t \mapsto \D_{r,v}Y_t.$ 
By \cite[Lemma 1]{StrickerYor} there exists 
a jointly in $(r,v,t,\om)$ measurable random map with the following property:  for each  $(r,v)$ this map has c\`adl\`ag paths and  is indistinguishable from the above  c\`adl\`ag version.
We assume now that  $\D_{r,v}Y_t$ is this measurable random map with c\`adl\`ag paths w.r.t.~$t.$
Then the pathwise limit  $\lim_{t \searrow r}\D_{r,v}Y_t$   is measurable in $(r,v,\om)$  
and the assertion follows by comparing the RHS of  \eqref{diffrep}  with \eqref{the-other-D-equation}.  \hfill  $\square$


\subsection{Example: A BSDE related to utility maximization}

In \cite{Becherer}  and  \cite{Morlais}  a class of BSDEs is considered which appears in exponential utility maximization. For these BSDEs an additional summand
arises in the generator which is only locally Lipschitz and is (in the simplest case) of the form: $[g^{\alpha}(U_s)]_\nu$ (see \eqref{g-nu})
with
\equa
g^{\alpha}(x) := \frac{e^{\alpha x} - \alpha x - 1}{\alpha} \quad \text{ for some }\alpha >0
\tion
and $g_1(x):=1$ for $x \in \R_0.$ Consider for $\ 0\leq t\leq T$ the following BSDE
\equal \label{beq3}
Y_t&=&\xi+\int_t^T \left(f_g\left(X,s,Y_{s}, Z_s, U_s \right) +[g^{\alpha}(U_s)]_\nu \right)ds-     \int_t^T Z_s   dW_s \non \\
&&  \quad \quad \quad \quad \quad \quad \quad \quad-\int_{{]t,T]}\times\R}U_s(x) \tilde N(ds,dx),
\tionl
where $f_g$ is defined like in \eqref{beq2}. Then we have the following assertion:

\begin{corollary}\label{diffcor}
Let $\xi \in \DD$ and assume that $\xi$ is a.s. bounded and  $\nu$ is a bounded measure.
If (${\A}_f$) is satisfied for $f_g$ and  if   there exists constants
$ K_1,K_2 >0$ such that for all $y,z,u \in \R$
\[ f(X,t,y,z,u) \le K_1+ K_2|y|
\]
for $\PP\otimes \lambda $ -a.a. $(t,\om) \in [0,T] \times  \Omega,$ 
then the following assertions hold for \eqref{beq3}.
\begin{enumerate}[(i)]
\item  For $\m$- a.e. $(r,v) \in [0,T]\times \R$  there exists a unique solution $(\mathcal{Y}^{r,v},\mathcal{Z}^{r,v}, \mathcal{U}^{r,v}) $ $\in  \mathcal{S}_2\times \Ltwo(W) \times \Ltwo(\tilde{N})$ to the BSDE
\equal \label{U-V-equation2}
\mathcal{Y}^{r,v}_t &=& \D_{r,v} \xi + \int_t^T  \left(F_{r,v}\kla s, \mathcal{Y}^{r,v}_s, \mathcal{Z}^{r,v}_s,  \mathcal{U}^{r,v} _s\mer + G_{r,v}(s,\mathcal{U}^{r,v} _s) \right)ds  \non\\
&& \quad \quad \quad     -\int_{{]t,T]}\times\R} \underline{\mathcal{Z}}^{r,v}_{s,x}M(ds,dx), \quad 0\le r \le t \le T \non \\
  \mathcal{Y}^{r,v}_s &=& \mathcal{Z}^{r,v}_s =   \mathcal{U}^{r,v}_s = 0,    \quad   0\le s< r\le T ,
\tionl
with $F_{r,v}$ and $\underline{\mathcal{Z}}^{r,v}_{s,x}$ given in Theorem \ref{diffthm} and

\equa
  G_{r,v} (s,u):= \left \{ \begin{array}{ll}
   \left[(e^{\alpha U_s}-1) u  \right]_\nu , & v=0,  \\
   &  \\
 \left[e^{\al U_s} g^{\alpha} (u) +  \frac{e^{\al U_s}-1}{\alpha} u \right]_\nu  ,    & v \neq 0.
\end{array} \right.
\tion

\item For the solution $(Y,Z,U)$ of \eqref{beq3} it holds
\begin{equation*}
Y, Z   \in \Ltwo([0,T];\DD), \quad U\in \Ltwo([0,T]\times\R_0;\DD),
\end{equation*}
and $\D_{r,y}Y$ admits a c\`adl\`ag version for $\m$- a.e. $(r,y) \in  [0,T]\times\R.$
\item $(\D Y,\D Z, \D U)$ is a version of $(\mathcal{Y},\mathcal{Z}, \mathcal{U}),$  i.e. for $\m$- a.e. $(r,v)$ it solves \eqref{U-V-equation2}.
\
\item  Setting  $D_{r,v}Y_r(\omega):= \lim_{t\searrow r}\D_{r,v}Y_t(\omega) $
for all $(r,v,\omega)$ for which  $ \D_{r,v}Y$ is c\`adl\`ag and $\D_{r,v}Y_r(\omega):=0$ otherwise, we have

\begin{align*} {\phantom{\int}}^p\left(\left(D_{r,0}Y_r\right)_{r\in{[0,T]}}\right) &\text{ is a version of  }(Z_{r})_{r\in{[0,T]}},\\
{\phantom{\int}}^p\left(\left(D_{r,v}Y_r\right)_{r\in{[0,T]}, v\in\R_0}\right) &\text{ is a version of  }(U_{r}(v))_{r\in{[0,T]}, v\in\R_0}.
\end{align*}
\end{enumerate}
\end{corollary}
\begin{proof}
Since $\xi$ is a.s. bounded, the L\'evy measure $\nu$ is finite and the generator satisfies the conditions of \cite[Theorem 3.5.]{Becherer}
it follows that $\|Y\|_{\mathcal{S}_\infty}  < \infty$  and
\begin{equation}\label{egger3}
  |U_s(x)|  \le 2\|Y\|_{\mathcal{S}_\infty}  \quad \text{ for   $\PP\otimes \lambda \otimes \nu$- a.e. }  (\om,s,x).
\end{equation}
From the fact that $g^{\alpha}$ is locally Lipschitz and $U$ is a.e. bounded it follows that $({\A}_f)$ (especially the Lipschitz condition) can be seen as satisfied  also for
$[g^{\alpha}(U_s)]_\nu:$

We find a $\mathcal{C}^1$ function $\widehat{g^{\alpha}}$ such that $g^\alpha = \widehat{g^{\alpha}}$ on $\left[-2\|Y\|_{\mathcal{S}_\infty},2\|Y\|_{\mathcal{S}_\infty}\right]$ and $$\mathrm{supp}(\widehat{g^{\alpha}})\subseteq \left[-3\|Y\|_{\mathcal{S}_\infty},3\|Y\|_{\mathcal{S}_\infty}\right].$$
Since by \eqref{egger3}, $g^\alpha(U_s(x))=\widehat{g^{\alpha}}(U_s(x))$, $\PP\otimes \lambda \otimes \nu$- a.e., it follows that for all $t\in {[0,T]}$
$$\int_t^T \!\!\!\! \left( f_g\left(X,s,Y_{s}, Z_s, U_s \right)\! +\![g^{\alpha}(U_s)]_\nu \right)ds\!=\!\int_t^T \!\!\!\!\! \left(f_g\left(X,s,Y_{s}, Z_s, U_s \right) \!+\![\widehat{g^{\alpha}}(U_s)]_\nu \right)ds,$$
$\PP$-a.s. So the solution of \eqref{beq3} also satisfies the BSDE with $g^\alpha$ replaced by $\widehat{g^{\alpha}}$ which satisfies the assumptions of Theorem 4.4.
\end{proof}

\subsection*{Acknowledgement} We would like to thank S. Geiss for  his helpful comments concerning the measurability needed in the proof of Theorem \ref{diffthm} \eqref{dalet}.


\appendix
\section{Appendix}

{\bf Proof of Lemma \ref{sprungaddthm}}\bigskip

{\tt Step 1.}
We have the a.s. representation of the L\'evy process $X$ as $$X_t=\gamma t+\sigma W_t + J_t.$$ We denote $B_t:=\gamma t+\sigma W_t$. Because of
$$\mathbb{P}\left(\left\{X\in \Lambda\right\}\right)=\int_{D{[0,T]}}\mathbb{P}_J\left(\Lambda-h\right) \mathbb{P}_B(dh),$$ we may restrict ourselves to 'pure
jump processes' (i.e. $X=J$).\bigskip

{\tt Step 2.}
Assume that $X$ is a  compound Poisson process. Then $\nu\left(\R_0\right)<\infty$.
We define  $\hat{\PP} :=\mathbb{P}\otimes \frac{\lambda\otimes\nu}{T\nu\left(\R_0\right)}$  on
$\big (\Omega\times[0,T]\times \R_0, \F\otimes\B([0,T] \times \R_0)\big ) $
and  $$\hat X_t(\omega,r,v):=X_t(\omega)+ \beta_t(r,v)$$ where   $\beta_t(r,v):=  v\one_{[r,T]}(t).$  By the law of total probability we get
\equal \label{totalprob}
 \hat{\PP}  \left(\hat X\in \Lambda\right)
\!=\!\sum_{k=0}^\infty \hat{\PP}\left(\hat X\in \Lambda \middle| N(]0,T]  \times \R_0)=k\right) \hat{\PP}
\left(N(]0,T]  \times \R_0)=k\right)\!.
\tionl
The conditional probabilities
$$\hat{\PP}\left(\hat X\in \Lambda \middle| N(]0,T]  \times \R_0)=k\right),  \quad k \in \N,$$
are the distributions of an independent sum of  $\beta$ and the   compound Poisson process $X$,  conditioned
on the event that the process $X$ jumps $k$ times in ${]0,T]}.$  The probability law of this conditioned  compound Poisson process is the same as the law of
a piecewise constant process which has exactly $k$ independent, uniformly distributed jumps in ${[0,T]}$  whose  jump sizes are independently identically distributed according to
$\frac{\nu}{\nu\left(\R_0\right)}$  and independent from the jump times. Therefore it holds
that
\equa \label{orderstat}
&& \les \hat{\PP}  \left(\hat X\in \Lambda \middle| N(]0,T]  \times \R_0)=k\right) \non\\
&\!=& \!\!\!\!\frac{(\lambda\otimes\nu)^{\otimes (k+1)}}{T^{k+1}\nu\left(\R_0\right)^{k+1}}\!\!\left(\!\!\left \{ \!\!\left((t_1,\!x_1),\dotsc,(t_k,\!x_k),(r,v)\right)\colon \!\sum_{l=1}^k x_l\one_{[t_l,T]}+
v\one_{[r,T]}\in \Lambda \right \}\!\!\right)  \non \\
&=&\!\!\!\mathbb{P}\left(X\in \Lambda \middle| N(]0,T]  \times \R_0) =k+1\right)=0,
\tion
where we used the argument concerning the distribution of a conditioned Poisson process again to come to the last line.
Hence, all summands of \eqref{totalprob} are zero, which shows the assertion for the special case of this step.\bigskip

{\tt Step 3.}
To extend the  second step to the case of a general pure-jump L\'evy process $X$ we split up $\R_0$ into sets $S_p$, $p\geq 1$ such
that $0<\nu(S_p)<\infty$. Without loss of generality set $S_1:=\left\{x\in\R : |x|>1\right\}$. We may assume that the sequence $(S_p)_{p\geq 1}$ is infinite, else we
would be in the compound Poisson case again.  From the proof of  the L\'evy-It\^o decomposition it follows that
$$X_t=\lim_{n\to \infty}\sum_{p=2}^n \left(X^{( p )}_t-t\int_{S_p}x\nu(dx)\right)+X^{(1)}_t,$$
where the convergence is $\mathbb{P}$-a.s., uniformly in $t\in{[0,T]}$ and the $(X^{( p )})$
given by
 \equa
  X^{( p )}_t=  \int_{[0,t] \times S_p} x N(ds,dx),
\tion
 are independent compound Poisson processes which have jumps distributed by
$\frac{\nu\mid_{S_p }}{\nu(S_p)}$.
Since for $k, p \in \N$ with   $p  \ge 1, $
$$0< (\PP \otimes \lambda\otimes \nu) \big ( \{ N(]0,T] \times S_p )=k \} \times [0,T]  \times S_p \big)<\infty,$$

we can  proceed in a similar way as in (\ref{totalprob}) for $\sigma$-finite measures: Let
\equa
      \overline{X}^{( p )}_t := X_t - X^{( p )}_t, \quad 0 \le t \le T,
\tion

an notice that $\overline{X}^{( p )}$ and  $X^{( p )}$ are independent. Then

\equal  \label{totalprob2}
&&\les(\PP \otimes \lambda\otimes \nu) \left(\hat X\in \Lambda\right)  \non \\
&=&\sum_{\genfrac{}{}{0pt}{}{p=1,}{k=0}}^\infty  (\PP \otimes \lambda\otimes \nu) \left(X +  \beta \in \Lambda \middle| \big \{N(]0,T] \times S_p )=k\big \}
\times [0,T]\times S_p \right)  \non \\
&& \quad \quad \quad \times (\PP \otimes \lambda\otimes \nu) \left( \big \{N(]0,T] \times S_p )=k\big \}
\times [0,T]\times S_p \right).
\tionl
 From Steps 1 and 2 we conclude that the summands  on the RHS of (\ref{totalprob2}) are zero again by
\begin{equation*}
\begin{split}
&(\PP \otimes \lambda\otimes \nu) \left(\overline{X}^{( p )}+X^{(p)} +  \beta \in \Lambda \middle| \big \{N(]0,T] \times S_p )=k\big \}
\times [0,T]\times S_p \right)  \\
&=(\PP \otimes \lambda\otimes \nu) \left(\overline{X}^{( p )}+X^{(p)}  \in \Lambda \middle| \big \{N(]0,T] \times S_p )=k+1\big \}
\times [0,T]\times S_p \right)\\
&=0,
\end{split}
\end{equation*}
which proves Step 3.\begin{flushright}\qed\end{flushright}

\begin{lemma}\label{gronwalllemma2}
Let $\left(g_n\right)_{n\geq 0}$ be a sequence of nonnegative numbers satisfying $g_0=0$ and
\begin{equation*}
g_{n+1}\leq \vare+C_n+\frac{1}{2}g_n,
\end{equation*}
where $\vare>0 $ and $\lim_{n\to \infty} C_n =0$.
Then it holds that
\begin{equation*}
\limsup_{n\to\infty}g_n\leq 2\vare.
\end{equation*}
Especially, if $C_n=0$ for all $n\in\N$, then
$g_n\leq 2\vare$
for all $n\in\N$.
\end{lemma}


\bibliographystyle{plain}

\begin{thebibliography}{99}
\bibitem{Alos}
E. Al\`os, J. A. Le\'on and J. Vives, {\it  An anticipating It\^o formula for L\'evy processes}, ALEA Lat. Am. J. Probab. Math. Stat. 4, 285-305, 2008.
\bibitem{Ankirch} S. Ankirchner, P. Imkeller, {\it Quadratic hedging of weather and catastrophe risk by using short term climate predictions}, HU Berlin, Preprint, 2008.
\bibitem{AnImRe} S. Ankirchner,  P. Imkeller and G. Dos Reis  {\it Classical and Variational Differentiability of BSDEs with Quadratic Growth}, Electron. J. Probab. Vol. 12 , no. 53, 1418-1453, 2007.
\bibitem{Applebaum} D. Applebaum, {\it L\'evy Processes and Stochastic Calculus}, Cambridge University Press, 2004.
\bibitem{bauer} H. Bauer, {\it Measure and Integration Theory}, de Gruyter, 2001.
\bibitem{bbp} G. Barles, R. Buckdahn, \'E. Pardoux, {\it Backward stochastic differential equations and integral-partial differential equations}, Stochastics Stochastics Rep. 60, no. 1-2, 57-83, 1997.
\bibitem{Becherer} D. Becherer, {\it Bounded solutions to backward SDEs with jumps for utility optimization and indifference hedging},  Ann. Appl. Probab. Vol.16, No. 4, 1733-2275, 2006.
\bibitem{Billing} P. Billingsley, {\it Convergence of probability measures}, John Wiley \& Sons, New York, 1968.
 \bibitem{Bogachev} V. Bogachev, {\it Gaussian measures}, AMS, 1998.
\bibitem{BouchardElie} B. Bouchard and R. Elie, {\it  Discrete time approximation of decoupled For-ward-Backward SDE with jumps},  Stochastic Process. Appl.  118, 53-75, 2008.
\bibitem{BouchardElieTouzi}B. Bouchard, R. Elie and N. Touzi, {\it Discrete-time approximation of BSDEs and probabilistic schemes for fully nonlinear PDEs},
In Advanced Financial Modelling. Radon Ser. Comput. Appl. Math. 8 91-124. de Gruyter, Berlin, 2009.
\bibitem{BouchardTouzi} B. Bouchard and N. Touzi,  {\it Discrete Time Approximation and Monte-Carlo Simulation of Backward Stochastic Differential Equations}, Stochastic Process.  Appl. 111, 175-206, 2004.
\bibitem{Delong} {\L}. Delong, {\it  Backward Stochastic Differential Equations with Jumps and Their Actuarial and Financial Applications}, Springer, 2013.
\bibitem{ImkellerDelong} {\L}. Delong, P. Imkeller, {\it On Malliavin's differentiability of BSDEs with time delayed generators driven by Brownian motions and Poisson random measures}, Stochastic Process. Appl. 120, 1748-1775, 2010.
\bibitem{delzeith} O. Delzeith, {\it On Skorohod spaces as universal sample path spaces},  arXiv:\\ math/0412092v1, 2004.
\bibitem{diestel} J. Diestel, and J. J. Uhl Jr., {\it Vector measures}, AMS 1977.
\bibitem{dinun} G. Di Nunno, B. \O ksendal, F. Proske, {\it Malliavin Calculus for L\'evy Processes with Applications to Finance}, Springer, 2009.
\bibitem{ElKarouiPengQuenez} N. El Karoui, S. Peng, and M.C. Quenez, {\it Backward Stochastic Differential Equations in Finance}, Math. Finance 7, 1-71, 1997.
\bibitem{ito} K. It\^o, {\it Spectral type of the shift transformation of differential process with stationary increments}, Trans. Amer. Math. Soc. 81, 253-263, 1956.
\bibitem{Janson} S. Janson, {\it Gaussian Hilbert Spaces}, Cambridge,  1997.
\bibitem{Kuo} H.-H. Kuo, {\it Gaussian measures in Banach spaces}, Bulletin of the American Mathematical Society 82, no. 5, 695-700, 1976.
\bibitem{LeeShih} Y. Lee and H. Shih, {\it Analysis of geralized L\'evy functionals}, J. Funct. Analysis, 211, 1-70, 2004.
\bibitem{Loekka2} A. L\o kka, {\it Martingale representation of functionals of L\'evy processes}, Stoch. Anal. Appl., 22, 867-892, 2004.
\bibitem{MaPoRe} T. Mastrolia, D. Possama\"i and A. R\'eveillac, {\it On the Malliavin differentiability of BSDEs,}  arXiv:1404.1026, 2014.
\bibitem{Meyer} P.\! A. Meyer, {\it Une remarque sur le calcul stochastique d\'ependant d'un param\`etre}, S\'eminaire de probabilit\'es (Strasbourg), tome 13, 199-203, 1979.
\bibitem{Morlais} M.-A. Morlais, {\it A new existence result for quadratic BSDEs with jumps with application to the utility maximization problem}, Stochastic Process. Appl.  120, 1966-1995, 2010.
\bibitem{Nualartbook} D. Nualart, {\it The Malliavin calculus and related topics. Second edition}, Probability and its Applications (New York). Springer-Verlag, Berlin, 2006.
\bibitem{PP1} \'E. Pardoux, S. Peng, {\it Backward Stochastic Differential Equations and Quasilinear Parabolic Partial Differential Equations}, Stochastic partial differential equations and their applications (Charlotte, NC, 1991),  200-217, Lecture Notes in Control and Inform. Sci., 176, Springer, Berlin, 1992.
\bibitem{PP2} \'E. Pardoux, S. Peng, {\it Adapted solution of a backward stochastic differential equation}, Systems Control Lett. 14, no. 1, 55-61, 1990.
\bibitem{Petrou} E. Petrou, {\it Malliavin Calculus in L\'evy spaces and Applications to Finance}, Electron. J. Probab., Vol. 13, no. 27, 852-879, 2008.
\bibitem{picard} J. Picard, {\it On the existence of smooth densities for jump processes}, P.T.R.F. 105, 481-511, 1996.
\bibitem{Steinicke} A. Steinicke, {\it  Functionals of a L\'evy Process on Canonical and Generic Probability Spaces},  J. Theoret. Probab., DOI: 10.1007/s10959-014-0583-7, 2014.
\bibitem{StrickerYor} C. Stricker and M. Yor, {\it Calcul stochastique d\'ependant d'un param\`etre}, Z. Wahrsch. Verw. Gebiete, 45(2), 109-133, 1978.
\bibitem{suv} J. Sol\'e, F. Utzet, J Vives, {\it Chaos expansions and Malliavin calculus for L\'evy processes}, Stoch. Anal. Appl., 595-612, Abel Symp., 2, Springer, Berlin, 2007.
\bibitem{suv2} J. Sol\'e, F. Utzet, J. Vives, {\it Canonical L\'evy process and Malliavin Calculus}, Stochastic Process. Appl. 117, pp. 165-187, 2007.
\bibitem{Sugita} H. Sugita, {\it On a characterization of the Sobolev spaces over an abstract Wiener space}, J. Math. Kyoto Univ. 25-4, 717-725, 1985.
\bibitem{Tangli} S. Tang and X. Li, {\it Necessary conditions for optimal control of stochastic systems with random jumps}, SIAM J. Control and Optim. 32, 1447-1475, 1994.
\end{thebibliography}

\end{document}